\documentclass[a4paper,english,11pt,leqno]{amsart}

\usepackage[utf8]{inputenc}
\usepackage{lmodern}
\usepackage[T1]{fontenc}
\usepackage[english]{babel}

\usepackage{geometry}
\usepackage{graphicx,epsfig}

\usepackage{enumitem}

\usepackage{amsmath,amsfonts,amssymb,amsthm,amsgen}
\usepackage{stmaryrd,mathrsfs,exscale,relsize}
\usepackage{leftindex}
\usepackage{tikz}
    \usetikzlibrary{cd}
\usepackage{color}

\usepackage[pagebackref=true]{hyperref}
    \hypersetup{pdfauthor={Pierre-Emmanuel Caprace, Geoffrey Janssens, and Fran\c{c}ois Thilmany}, bookmarksopen, bookmarksnumbered, colorlinks={true}, linkcolor={gray}, citecolor={gray}, urlcolor={gray}}
\usepackage[english]{cleveref}

\newtheorem{thm}{Theorem}[section]
\newtheorem{cor}[thm]{Corollary}
\newtheorem{lem}[thm]{Lemma}
\newtheorem{prop}[thm]{Proposition}

\newtheorem{ques}[thm]{Question}
\theoremstyle{definition}
\newtheorem{exe}[thm]{Example}
\newtheorem{defn}[thm]{Definition}
\newtheorem{rem}[thm]{Remark}

\newtheoremstyle{step}
    {6pt} % Space above
    {6pt} % Space below
    {\itshape} % Body font
    {} % Indent amount
    {\bfseries} % Theorem head font
    {.} % Punctuation after theorem head
    {.5em} % Space after theorem head
    {} % Theorem head spec (can be left empty, meaning `normal')
\theoremstyle{step}
\newtheorem*{mainclaim*}{Main Claim}
\newtheorem*{claim*}{Claim}
\newcounter{claimcount}
\newtheorem{claim}[claimcount]{Claim}
    \crefname{claim}{claim}{claims}
\newcounter{stepcount}
\newtheorem{step}[stepcount]{Step}
    \crefname{step}{step}{steps}
\newcounter{casecount}
\newtheorem{case}[casecount]{Case}
    \crefname{case}{case}{cases}

\Crefname{cor}{Corollary}{Corollaries}

\numberwithin{equation}{section}

\newcommand{\F}{\mathbb{F}}

\newcommand{\Z}{\mathbb{Z}}
\newcommand{\Q}{\mathbb{Q}}
\newcommand{\R}{\mathbb{R}}

\newcommand{\rA}{\mathrm{A}}

\newcommand{\rH}{\mathrm{H}}
\newcommand{\rK}{\mathrm{K}}

\newcommand{\rS}{\mathrm{S}}

\newcommand{\rZ}{\mathrm{Z}}

\newcommand{\llag}{\langle\!\langle}
\newcommand{\rrag}{\rangle\!\rangle}
\newcommand{\ol}{\overline}

\newcommand{\restr}[2]{ %
	\sbox0{\ensuremath{#1}}
	\ensuremath{
	{#1}\raisebox{-.9ex}{%
		\(\mkern2mu {\vrule height 2ex} \mkern2mu {\scriptstyle{#2}}\)}
	}}

\newcommand{\qker}{\operatorname{qker}}
\newcommand{\Hom}{\operatorname{Hom}}

\newcommand{\Gal}{\operatorname{Gal}}
\newcommand{\GL}{\operatorname{GL}}
\newcommand{\PGL}{\operatorname{PGL}}
\newcommand{\SL}{\operatorname{SL}}
\newcommand{\Soc}[1][]{\operatorname{Soc}^{#1}}
\newcommand{\SocA}{\operatorname{SocA}}
\newcommand{\SocH}{\operatorname{SocH}}
\newcommand{\ind}{\operatorname{Ind}}

\newcommand{\core}{\operatorname{core}}

\DeclareEmphSequence{\itshape,\upshape\bfseries,\itshape\bfseries}
\newcommand{\emphb}[1]{\emph{\emph{#1}}}
\newcommand{\ee}{1}
\newcommand{\Fi}{F}
\newcommand{\zz}{z}
\newcommand{\muF}{\mu_\Fi}
\newcommand{\muz}{\mu_\zz}

\title[Center-preserving irreducible representations of finite groups]{Center-preserving irreducible representations\\ of finite groups}

\date{10 March 2026}

\author{Pierre-Emmanuel Caprace}
    \address{Pierre-Emmanuel Caprace: UCLouvain --- IRMP, Chemin du Cyclotron 2, box L7.01.02, B-1348 Louvain-la-Neuve. }
    \email{pe.caprace@uclouvain.be}

\author{Geoffrey Janssens}
    \address{Geoffrey Janssens: UCLouvain --- IRMP, Chemin du Cyclotron 2, B-1348 Louvain-la-Neuve, and Vrije Universiteit Brussel --- Department of Mathematics and Data Science, Pleinlaan $2$, 1050 Elsene. }
	\email{geoffrey.janssens@uclouvain.be}

\author{Fran\c{c}ois Thilmany}
    \address{François Thilmany: KU Leuven --- Departement Wiskunde, Celestijnenlaan 200B, B-3001 Leuven. }
    \email{francois.thilmany@kuleuven.be}

\thanks{P.-E.\ C.\ acknowledges support from the FWO and F.R.S.-FNRS under the Excellence of Science (EOS) program (project ID 40007542). 
G.\ J.\ and F.\ T.\ are grateful to the FWO (grants 2V0722N \& 1221722N) and to the F.R.S.-FNRS (grants 1.B.239.22 \& FC 4057) for their financial support.}

\begin{document}

\begin{abstract}
Given finite groups $H \leq G$, a representation $\sigma$ of $G$ is called center-preserving on $H$ if the only elements of $H$ that become central under $\sigma$ are those that were already central in $G$. 
We prove that if $H$ has a faithful irreducible representation $\rho$, then at least one of the irreducible components of the induction $\operatorname{Ind}_H^G(\rho)$ is center-preserving on $H$. 
In consequence, $H$ has a faithful irreducible representation if and only if every finite group $G$ containing $H$ as a subgroup has an irreducible representation whose restriction to $H$ is faithful, and which is center-preserving on $H$. 
In addition, we give examples illustrating the sharpness of the statement, and discuss the connection with projective representations. 
\end{abstract}

\subjclass{20C15 (Primary), 20C05, 20C25 (Secondary)}

\maketitle

%%%%% Body %%%%%

\section{Introduction}

The study of faithful representations is a classical topic, that can be traced back all the way to Frobenius' original work on characters of finite groups. 
The question whether a given finite group admits a faithful irreducible representation can be attributed to Burnside (see \cite[Note F]{Burnside11}). 
Schur's Lemma \cite{Schur05} gave a first constraint for the existence of such a representation: a finite group which admits a faithful irreducible representation must have cyclic center; 
Fite \cite{Fite06} showed that for nilpotent groups these two conditions are in fact equivalent. 
The 20\textsuperscript{th} century subsequently produced three main criteria for the existence of a faithful irreducible representation, due respectively to Akizuki, Weisner, and Gasch\"utz, each one involving the socle of the group. 
Perhaps most popular among them is Gasch\"utz' criterion, which we will recall in \Cref{sec:Gaschutz} below. 
We refer the reader interested in additional background to the precise historical account provided in \cite[§2]{Szechtman16}. 
\medskip

The goal of the present note is to initiate the study of a different but related class of (irreducible) representations, that we named \emph{center-preserving}. 
The notion is naturally defined for any group homomorphism. 

\begin{defn} \label{def:center-preserving}
Let $\sigma \colon G \to L$ be a group homomorphism. 
We call
\[
\rZ(\sigma)= \sigma^{-1}\big(\rZ(\sigma(G))\big)
\]
the \emphb{center of $\sigma$}. 
We say that $\sigma$ is \emphb{center-preserving} when $\rZ(\sigma) = \rZ(G)$. 
Thus $\sigma \colon G \to L$ is {center-preserving} if and only if the kernel of $\sigma$ is central and the center of its image coincides with the image of the center of $G$, that is, if and only if $\ker(\sigma) \leq \rZ(G)$ and $\rZ(\sigma(G)) = \sigma(\rZ(G))$. 

Let now $H$ be a subgroup of $G$. 
We say that $\sigma$ is \emphb{center-preserving on $H$} if 
\[
\rZ(\sigma) \cap H \leq \rZ(G).
\]
In particular, $\sigma$ is center-preserving on $G$ if and only if it is center-preserving. 
In that case, $\sigma$ is center-preserving on every subgroup. 
\end{defn}

Of course, any injective homomorphism from $G$ is center-preserving, and by definition any center-preserving homomorphism has a central kernel. 
However, a homomorphism with central kernel need not be center-preserving, and a center-preserving homomorphism need not be injective. 
Moreover, a homomorphism $\sigma \colon G \to L$ whose restriction to $H$ is faithful, need not be center-preserving on $H$ (unless $H$ is normal, see \Cref{rem:centerpreservingnormalsubgroup}). 
\smallskip

The following conventions  will be in effect throughout the entire paper. 
\smallskip

\noindent \textbf{Notation.} 
Throughout, $G$ denotes a finite group and $\Fi$ is a field of cross-characteristic, that is, the characteristic of $\Fi$ does not divide the order of $G$. 
A finite-dimensional, linear representation over $\Fi$ of an arbitrary group $\mathscr G$ will simply be called a representation of $\mathscr G$. 
The center of $\mathscr G$ is denoted $\rZ(\mathscr G)$, and given a representation $\sigma$ of $\mathscr G$, we denote its center by $\rZ(\sigma)$. 
As defined above, $\rZ(\sigma)$ is the preimage under $\sigma$ of the center of $\sigma(\mathscr G)$.\footnote{
In view of Schur's lemma, when $\sigma$ is an irreducible representation of $G$ and $\Fi$ contains all $|G|$\textsuperscript{th} roots of unity, $\rZ(\sigma)$ is characterized as the normal subgroup of $G$ consisting of those $g \in G$ such that $\sigma(g)$ is a scalar operator. 
In this case, it coincides with the \emph{quasikernel} (or \emph{projective kernel}) of $\sigma$, that is, the kernel of the projectivization of $\sigma$. 
%In general, $\rZ(\sigma)$ is the normal subgroup of $G$ consisting of those $g \in G$ such that $\sigma(g)$ lies in the center of the commutant of $\sigma(G)$, a product of finite extensions of $\Fi$. 
}

\smallskip
The main result of this note is the following. 

\begin{thm} \label{thm:faithfulcenterpreservingirrep}
Let $H \leq G$ be finite groups. 
Suppose that $H$ has a faithful irreducible representation $\rho$. 
Then one of the irreducible constituents of $\ind_H^G(\rho)$ is center-preserving on $H$. 
\end{thm}

It can happen that $\ind_H^G(\rho)$ has multiple irreducible components, among which exactly one is center-preserving on $H$ (see \Cref{exe:exactly-one}). 

\smallskip
While it is clear that any faithful representation of $G$ is center-preserving (on every subgroup), the converse clearly fails. 
For example, every representation of an abelian group is center-preserving, but unless this abelian group is cyclic, none of its irreducible representations are faithful. 
We shall provide a characterization of the finite groups admitting a center-preserving irreducible representation in \Cref{cor:centerpreserving}. 

For now, we record the following corollary of \Cref{thm:faithfulcenterpreservingirrep}, providing a new characterization of the finite groups that admit a faithful irreducible representation. 

\begin{cor}\label{cor:faithfulcenterpreservingirrep}
 For a finite group $H$, the following conditions are equivalent. 
\begin{enumerate}[leftmargin=2em,itemsep=.2ex,label=\textup{(\roman*)}]
    \item $H$ has a faithful irreducible representation.
    \item For any injective homomorphism $\iota \colon H \to G$, the finite group $G$ has an irreducible representation $\sigma$ such that $\sigma \circ \iota $ is injective and $\sigma$ is center-preserving on $\iota(H)$. 
\end{enumerate}   
\end{cor}

The implication (ii) $\implies$ (i) is immediate, by taking $\iota$ the identity map on $H$. 
The converse is a direct consequence of \Cref{thm:faithfulcenterpreservingirrep}. 
Indeed, by Frobenius reciprocity, the restriction $\restr{\sigma}{H}$ of any irreducible constituent $\sigma$ of $\ind_H^G(\rho)$ contains $\rho$, and is therefore faithful on $H$. 

\begin{rem} \label{rem:specialcasemaintheorem}
This last observation can in fact be used to prove a special case of \Cref{thm:faithfulcenterpreservingirrep} as follows. 
If $\sigma$ is a representation of $G$ whose restriction to $H$ is faithful, then $\rZ(\sigma) \cap H \leq \rZ(H)$. 
By Frobenius reciprocity, this holds in particular if $\sigma$ is a constituent of $\ind_H^G(\rho)$ with $\rho$ faithful irreducible. 
Since $H$ has a faithful irreducible representation by hypothesis, $\rZ(H)$ must be a cyclic group. 
In the special case where the subgroups of $\rZ(H)$ are totally ordered by inclusion (equivalently, when the order of the cyclic group $\rZ(H)$ is a prime power), it is not difficult to verify that any irreducible constituent of $\ind_H^G(\rho)$ that minimizes $\rZ(\sigma) \cap H$ is actually center-preserving on $H$. 
Thus, in that special case, \Cref{thm:faithfulcenterpreservingirrep} may be viewed as a fairly direct consequence of Frobenius reciprocity. 
The proof we give for the general case involves a much more elaborate argument (and many more invocations of Frobenius reciprocity). 
\end{rem}

\begin{rem}
Upon reading the statement of \Cref{thm:faithfulcenterpreservingirrep}, one could expect that if $H$ has a center-preserving irreducible representation, then any finite group $G$ containing $H$ as a subgroup has an irreducible representation which is center-preserving on $H$. 
This is however not true, as we shall see in \Cref{exe:D4D4D4} below. 
\end{rem}

In the last section, we will discuss the connection between center-preserving linear representations and faithful projective representations. 
From \Cref{thm:faithfulcenterpreservingirrep}, one can namely deduce existence results for irreducible projective representations that are faithful on a subgroup. 
Here is one such corollary (see also \Cref{cor:centerpreservingprojirrep2} for a more general statement). 

\begin{cor}[of \Cref{thm:faithfulcenterpreservingirrep}] \label{cor:centerpreservingprojirrep}
Let $H \leq G$, and suppose that $H$ admits a faithful irreducible representation. 
Then there exists an irreducible $1$-projective representation $\ol{\sigma}$ of $G$ such that $\ker(\ol{\sigma}) \cap H \leq \rZ(G)$. 

In particular, if $\rZ(G) \cap H = \{\ee\}$, then $\ol{\sigma}$ is an irreducible $1$-projective representation of $G$ which is faithful on $H$. 
\end{cor}

Indeed, the center $\rZ(\sigma)$ of a representation $\sigma$ of $G$ always contains its quasikernel $\qker(\sigma)$. 

\medbreak

Dynamics on projective spaces are a valuable tool in geometric group theory. 
That purpose was in fact our original motivation for studying center-preserving representations: with the ultimate aim to construct free products in the group of units $R[G]^\times$ of the group ring of $G$ over a subring $R$ of $F$, representations of $G$ that are center-preserving on a subgroup $H$ arose naturally in \cite{JTT25} as a mean to ensure that the corresponding projective representation of $G$ has smallest possible kernel when restricted to $H$. 
For example, using the machinery developed in \cite{JTT25}, the following result is a direct consequence of \Cref{cor:centerpreservingprojirrep}. 

\begin{cor}[Corollary 5.15 in {\cite{JTT25}}] \label{cor:existenceamalgamarbitraryfield}
Suppose that $\Fi$ has characteristic $0$. 
Let $H \leq G$ be a subgroup possessing a faithful irreducible representation. 
Set $C = H \cap \rZ(G)$. 
The set of regular semisimple elements $\gamma \in \Fi[G]^\times$ of infinite order with the property that the canonical map
\[
\langle \gamma, C \rangle \ast_{C} H \to \langle \gamma, H \rangle
\]
is an isomorphism, is Zariski-dense in the multiplicative group of $\Fi[G]$. 
\end{cor}

More precise applications of \Cref{thm:faithfulcenterpreservingirrep} are provided in \cite{JTT25}. 
Namely, for $G$ not a Dedekind group, \cite[Theorem~5.3]{JTT25} establishes the existence of irreducible representations of $G$, center-preserving on $H$, and for which the image of $G$ is not a Frobenius complement. 
This is then used to find ping-pong partners for $H$ among the bicyclic units of the group ring $\Fi[G]$, see \cite[Theorem~6.9]{JTT25}. 
\medbreak

\noindent \textbf{Outline.} 
Our approach to proving \Cref{thm:faithfulcenterpreservingirrep} is of group theoretic nature, rather than representation theoretic. 
A key ingredient is Gasch\"utz' characterization of the finite groups possessing a faithful irreducible representation \cite{Gaschutz54}, whose statement and proof are recalled in \Cref{sec:Gaschutz}, as well as some of its avatars \cite{BH08,CH20}. 

Without further ado, \Cref{subsec:proof} takes us to the proof of \Cref{thm:faithfulcenterpreservingirrep}, and in \Cref{subsec:examplescpr}, we give a few examples highlighting the logical independence between faithfulness and center-preservingness on a subgroup. 

Last but not least, we discuss in \Cref{sec:fpr} the projective counterparts of center-preserving linear representations. 
After a brief reminder (\Cref{subsec:projectiverepresentations}), we give a criterion for the existence of a center-preserving irreducible representation in \Cref{subsec:absolutefpr}, and relate it to the existence of the newly introduced $c$-faithful $c$-projective representations. 
After the examples of \Cref{subsec:examplesfpr}, we conclude this note with a couple of consequences of \Cref{thm:faithfulcenterpreservingirrep} concerning irreducible projective representations that are faithful on a subgroup (see \Cref{subsec:relativefpr}), opening the door to their study.

\section{Groups with a faithful irreducible representation} \label{sec:Gaschutz}

In order to formulate Gasch\"utz' Theorem, we recall that the \emphb{socle} of a finite group $G$, denoted $\Soc(G)$, is the subgroup of $G$ generated by all the minimal normal subgroups of $G$. 
The \emphb{abelian part of the socle}, denoted by $\SocA(G)$, is the subgroup of $G$ generated by all the minimal normal subgroups of $G$ that are abelian. 
The subgroup $\SocA(G)$ is indeed abelian, and is a direct factor of $\Soc(G)$ (see \Cref{lem:minimal-normal}). 

\begin{thm}[{See \cite{Gaschutz54}, \cite[Theorem~2]{BH08}}] \label{thm:Gaschutz}
For a finite group $G$, the following conditions are equivalent. 
\begin{enumerate}[leftmargin=2em,itemsep=.2ex,label=\textup{(\roman*)}]
\item \label{item:Gaschutz1} $G$ has a faithful irreducible representation. 
\item \label{item:Gaschutz2} $\Soc(G)$ is generated by a single conjugacy class. 
\item \label{item:Gaschutz3} $\SocA(G)$ is generated by a single conjugacy class. 
\item \label{item:Gaschutz4} Every normal subgroup of $G$ contained in $\SocA(G)$  is generated by a single conjugacy class. 
\end{enumerate}
\end{thm}

Note that Gasch\"utz \cite{Gaschutz54} originally considered the base field $\Fi$ to be algebraically closed. 
Nonetheless, the proof he gives only relies on the semisimplicity of the group algebra $\Fi[G]$ and on the good behavior of the canonical maps $\Fi[G] \leftrightarrows \Fi[G/N]$, for $N \triangleleft G$.
His result is thus equally valid (with the same proof) over a field $\Fi$ whose characteristic does not divide $|G|$. 

Alternatively, observe that $G$ has a faithful irreducible representation over $\Fi$ if and only if $G$ has a faithful irreducible representation over any algebraic closure $\bar{\Fi}$ of $\Fi$. 
This is a consequence of the fact that two $\bar{\Fi}$-representations of $G$ belonging to the same $\Gal(\bar{\Fi}/\Fi)$-orbit (of $\bar{\Fi}$-representations up to equivalence) have the same kernel. 
Indeed, if $\rho$ is an irreducible $\Fi$-representation of $G$, the extension $\rho_{\bar{\Fi}}$ of $\rho$ to $\bar{\Fi}$ decomposes into a sum of irreducible $\bar{\Fi}$-representations belonging to a single orbit of $\Gal(\bar{\Fi}/\Fi)$, and each irreducible $\bar{\Fi}$-representation $\sigma$ of this orbit occurs in the decomposition of $\rho_{\bar{\Fi}}$ with the same multiplicity $m$ (equal to the Schur index of $\sigma$ over $\Fi$ by definition). 
Thus $\ker(\sigma) = \ker(\rho_{\bar{\Fi}}) = \ker(\rho)$ is trivial as soon as $\rho$ is faithful. 
Conversely, if $\sigma$ is an irreducible $\bar{\Fi}$-representation of $G$, let $\rho$ be an irreducible $\Fi$-representation of $G$ whose extension $\rho_{\bar{\Fi}}$ to $\bar{\Fi}$ contains $\sigma$. 
Then also this way around, $\ker(\rho) = \ker(\rho_{\bar{\Fi}}) = \ker(\sigma)$ is trivial when $\sigma$ is faithful. 

We refer the reader to \cite[{IX} §4]{BKZ18} for a modern treatment of the various criteria concerning the existence of faithful irreducible representations, and to \cite[XXIII]{BKZ19} for the behavior of representations under field extensions. 
For the convenience of the reader, we include here a slightly different proof of \Cref{thm:Gaschutz} in the present setting, inspired by \cite{BH08}. 

\begin{proof}[Proof of \Cref{thm:Gaschutz}]
Recall that the Pontryagin dual of a finite abelian group $N$ is the group $\widehat{N} = \Hom(N, \Q / \Z)$ of homomorphisms $N \to \Q / \Z$, endowed with the operation of point-wise multiplication. 
As $\Fi$ is a field of cross-characteristic, every homomorphism $\chi \colon N \to \Q / \Z$ factors through the torsion subgroup of the multiplicative group $\bar{\Fi}^\times$ of any algebraic closure $\bar{\Fi}$ of $\Fi$, so that we can (and will) identify $\widehat{N}$ with $\Hom(N, \bar{\Fi}^\times)$. 

\smallskip
Set now $N = \SocA(G)$. 
Suppose that $\rho$ is a faithful irreducible $\Fi$-representation of $G$. 
As we have seen above, any irreducible constituent $\sigma$ of its extension $\rho_{\bar{\Fi}}$ to the algebraic closure $\bar{\Fi}$ is a faithful irreducible $\bar{\Fi}$-representation of $G$. 
Let $\restr{\sigma}{N} = n_1 \chi_1 \oplus \dots \oplus n_k \chi_k$ be a decomposition into irreducibles of the restriction of $\sigma$ to $N$. 
As $N$ is abelian, each $\chi_i$ can now be seen as an element of the Pontryagin dual $\widehat N$ of $N$. 
By Clifford's theorem, the $G$-orbit of $\chi_1$ in the dual $\widehat N$ contains all the $\chi_i$'s. 
Since moreover $\restr{\sigma}{N}$ is faithful, it follows that the dual $\widehat N$ is generated by a single $G$-orbit, that of $\chi_1$. 
By duality (see \cite[Lemma~14]{BH08}), we infer that $N$ is generated by a single conjugacy class. 
This proves \ref{item:Gaschutz1} $\implies$ \ref{item:Gaschutz3}. 

For the converse, suppose that $N = \SocA(G)$ is generated by a single conjugacy class. 
Hence $\widehat N$ is generated by a single $G$-orbit, by duality. 
Let $\chi$ be a representative of that orbit, and view $\chi$ as the one-dimensional $\bar{\Fi}$-representation of $N$ given by sending $n \in N$ to the multiplication by $\chi(n)$. 
Let also $\pi$ be a faithful irreducible $\bar{\Fi}$-representation of the direct complement $\SocH(G)$ of $\SocA(G)$ in $\Soc(G)$, product of the non-abelian minimal normal subgroups of $G$. 
The representation $\pi$ exists because $\SocH(G)$ is a direct product of non-abelian finite simple groups. 
Let now $\sigma$ be any irreducible component of the induced representation $\ind_{\Soc(G)}^G(\chi \otimes \pi)$. 
It follows from Frobenius reciprocity that $\restr{\sigma}{\Soc(G)}$ is faithful, since it contains the irreducible representation $\chi^g \otimes \pi^g$ for all $g \in G$. 
Hence the kernel of $\sigma$ does not contain any minimal normal subgroup of $G$. 
Therefore $\sigma$ is a faithful $\bar{\Fi}$-representation of $G$, and any irreducible $\Fi$-representation whose extension to $\bar{\Fi}$ contains $\sigma$ is then a faithful irreducible $\Fi$-representation of $G$. 

The equivalence $\ref{item:Gaschutz2} \iff \ref{item:Gaschutz3}$ is a straightforward consequence of the product structure of the socle: $\Soc(G) = \SocA(G) \times \SocH(G)$. 
If the conjugacy class of $g$ generates $\Soc(G)$ then its projection to $\SocA(G)$ generates $\SocA(G)$. 
Conversely, if the conjugacy class of $g_0$ generates $\SocA(G)$, it suffices to write $\SocH(G)$ as a product of non-abelian simple groups $S_1, \dots, S_l$, and pick $g_i \in S_i \setminus \{\ee\}$ to see that the conjugacy class of $g_0 g_1 \cdots g_l$ generates $\Soc(G)$. 

Lastly, $\ref{item:Gaschutz3}$ implies $\ref{item:Gaschutz4}$ for a similar reason: by the complete reducibility of the socle, for every normal subgroup $A$ of $G$ contained in $\SocA(G)$, there is a normal subgroup $A'$ of $G$ for which $\SocA(G) = A \times A'$.
If the conjugacy class of $g_0$ generates $\SocA(G)$, its projection onto $A$ generates $A$. 
\end{proof}

\section{Irreducible representations that are center-preserving on a subgroup}

\subsection{Towards the proof of \texorpdfstring{\Cref{thm:faithfulcenterpreservingirrep}}{Theorem \ref{thm:faithfulcenterpreservingirrep}}} \label{subsec:proof}

We shall need the following consequence of Goursat's lemma. 

\begin{lem} \label{lem:Goursat}
Let $H \leq G$. 
Let also $M$, $N$ be normal subgroups of $G$ with $M \cap H = N \cap H = M \cap N = \{\ee\}$, and suppose that $G = HN$. 
Then there is an $H$-equivariant injective homomorphism $M \to N$. 
In particular, if $|M| = |N|$, then $M \cong N$. 
\end{lem}
\begin{proof}
Let $M' \leq H$ be the image of $M$ under the canonical projection $G \to  G/N \cong H$. 
By construction, $M \leq M' \ltimes N$ and $M' \leq MN$. 
Because $M \cap N = \{\ee\}$, we have $M N \cong M \times N$. 
Let $N' \leq N$ be the image of $M'$ under the canonical projection $MN \to N$. 
Thus $M'$ is a subgroup of the direct product $MN' \cong  M \times N'$. Since $H \cap M =\{\ee\} = H \cap N$ by assumption, we have $M' \cap M = \{\ee\} = M' \cap N'$. 
It then follows from Goursat's lemma that $M'$ is the graph of an isomorphism $M \to N'$. 
The $H$-equivariance follows from the fact that $M$, $M'$ and $N$ are all normalized by $H$. 
\end{proof}

Recall that minimal normal subgroups of a finite group are characteristically simple, hence split as direct products of pairwise isomorphic simple groups. 
We shall also need the following basic fact concerning minimal normal subgroups. 

\begin{lem} \label{lem:minimal-normal}
Let $G$ be a finite group, let $\mathscr M_\rA$ (resp. $\mathscr M_\rS$) be a set of abelian (resp. non-abelian) minimal normal subgroups of $G$. 
Set $W_\rA = \langle M \mid M \in \mathscr M_\rA \rangle$ and $W_\rS = \langle M \mid M \in \mathscr M_\rS \rangle $. 
Then $W_\rA W_\rS \cong W_\rA \times W_\rS$, and $W_\rS$ splits as the direct product of the elements of $\mathscr M_\rS$. 
In particular, the only minimal normal subgroups of $G$ contained in $W_\rS$ are the elements of $\mathscr M_\rS$. 
\end{lem}
\begin{proof}
This follows from \cite[Proposition~1]{BH08}. 
\end{proof}

It will be convenient to view abelian normal subgroups of $G$ as $\Z[G]$-modules. 
Each abelian minimal normal subgroup is a simple $\Z[G]$-module. 
Thus their join $\SocA(G)$ is a semisimple $\Z[G]$-module, and every submodule of $\SocA(G)$ is a direct summand (see \cite[Proposition~2.14]{CH20}).
It is also useful to keep in mind that a normal subgroup $V$ of $G$ contained in $\SocA(G)$ is generated by a single conjugacy class if and only if it is cyclic as a $\Z[G]$-module (see \cite[Proposition~2.15]{CH20}). 

\begin{lem} \label{lem:semidirect}
Let $H \leq G$ be finite groups, and suppose that $H$ has a faithful irreducible representation $\rho$. 
Let also $N$ be a non-abelian minimal normal subgroup of $G$, and suppose that $N \cap H =\{\ee\}$ and $G = HN$. 
Then $G$ has a faithful irreducible representation $\sigma$ whose restriction to $H$ contains $\rho$. 
\end{lem}   

Observe that \Cref{lem:semidirect} is wrong without the hypothesis that the minimal normal subgroup $N$ is non-abelian: $G = C_p \times C_p$ and $H \cong N \cong C_p$ provides a counterexample. 

On the other hand, if $N$ is non-abelian, then $\SocA(G) \cap N = \{\ee\}$, so that $\SocA(G) $ maps injectively to a subgroup of $\SocA(H)$ under the canonical projection $G \to G/N \cong H$. 
Since $H$ has a faithful irreducible representation by hypothesis, it follows from Gasch\"utz' \Cref{thm:Gaschutz} that $\SocA(H)$ is a cyclic $\Z[H]$-module, hence $\SocA(G)$ is a cyclic module as well, so that a second application of \Cref{thm:Gaschutz} ensures that $G$ has a faithful irreducible representation. 
Establishing that a faithful irreducible representation of $G$ may be found among the irreducible constituents of $\ind_H^G(\rho)$ requires a more detailed analysis. 

\begin{proof}[Proof of \Cref{lem:semidirect}]
Let $\mathscr M$ be the set of all minimal normal subgroups $M$ of $G$ with $M \cap H =\{\ee\}$, let $\mathscr M_\rA \subseteq \mathscr M$ be the subset consisting of the abelian ones, let $\mathscr M_\rS = \mathscr M \setminus \mathscr M_\rA$, and let $\mathscr M_\rS^0 = \{M \in \mathscr M_\rS \mid M \neq N\}$. 
Let also $\mathscr M'$ be the set of all minimal normal subgroups $M$ of $G$ with $M \cap H \neq \{\ee\}$. 

We set 
\[
W_\rA = \langle M \mid M \in \mathscr M_\rA \rangle, \qquad W_\rS = \langle M \mid M \in \mathscr M_\rS^0 \rangle, \quad \text{and} \quad W = W_\rA W_\rS, 
\]
so that $W_\rA$, $W_\rS$ and $W$ are normal subgroups of $G$.

First, we claim that $W \cap N = \{\ee\}$. 
Indeed, suppose the contrary; then $N \leq W$ by the minimality of $N$. 
The group $W$ splits as a direct product $W_\rA \times W_\rS$ (see \Cref{lem:minimal-normal}), and the image of $N$ under the canonical projection $W \to W_\rA$ is trivial since $W_\rA$ is abelian and $N$ is not. 
Thus $N \leq W_\rS$, which is impossible since the only minimal normal subgroups of $G$ contained in $W_\rS$ are the elements of $\mathscr M_\rS^0$ by \Cref{lem:minimal-normal}.

\medskip
Next, we set $H' = HW$ and distinguish two cases. 

\begin{case}
    $H'\neq G$.
\end{case}     

Let $\pi \colon G \to G/N \cong H$ be the canonical projection, and let $\rho'$ denote the irreducible representation of $H'$ obtained by restriction of $\rho \circ \pi$. 
The restriction $\restr{\pi}{H}$ is the identity on $H$, hence we have $\restr{\rho'}{H} = \rho$. 
Moreover, $\rho'$ is injective on $W$ by the claim. 
In particular $\restr{\rho'}{M}$ is non-trivial on every $M \in \mathscr M \setminus \{N\}$.

Recall that the kernel of the induced representation $\ind_{H'}^G(\rho')$ is the subgroup $\core_G(\ker(\rho'))$ of $H'$. 
Since by assumption $H' \neq G$, we have $N \not \leq H'$. 
Therefore, there must be an irreducible constituent $\sigma$ of $\ind_{H'}^G(\rho')$ which is non-trivial on $N$. 
Observe that $\restr{\sigma}{H'}$ contains $\rho'$ by Frobenius reciprocity. 
We deduce that $\sigma$ is non-trivial on every $M \in \mathscr M$. 
Since $\restr{\rho'}{H} = \rho$ is injective by hypothesis, it  follows that $\sigma$ is also non-trivial on every minimal normal subgroup $M \in \mathscr M'$. 
Altogether, $\sigma$ is non-trivial on every minimal normal subgroup of $G$. 
Thus $\sigma$ is faithful, concluding the proof in this case. 

\begin{case}
     $H' = G$. 
\end{case}

Since $W \cap N$ is trivial by the claim above, the group $N$ acts trivially on $W$, hence on $W_\rA$. 
As $G/N \cong H$, it follows that the abelian group $W_\rA$, which is a semisimple $\Z[G]$-module, is also a semisimple $\Z[H]$-module. 
The intersection $H \cap W_\rA$ is a submodule of $W_\rA$, hence we may pick a complement $W'_\rA$, so that $W_\rA \cong W'_\rA \times (H \cap W_\rA)$. 
We set $W' = W'_\rA W_\rS$. 
Thus we have 
\[
G = H' = HW = HW'.
\]

We claim that $H \cap W' = \{\ee\}$. 
Indeed, $H \cap W'$ is a normal subgroup of $H$ which commutes with $N$ (since $W \cap N = \{\ee\}$ by the previous claim). 
In consequence, $H \cap W'$ is normal in $G$. 
Let $M$ be a minimal normal subgroup of $G$ contained in $H \cap W'$. 
If $M$ is abelian, then $M  \leq H \cap W'_\rA$, which is impossible by the definition of $W'_\rA$. 
If $M$ is non-abelian, then the image of $M$ under the canonical projection $W'\cong W'_\rA \times W_\rS \to W'_\rA$ is trivial, so that $M \leq W_\rS$. 
This is also impossible: by \Cref{lem:minimal-normal} the only minimal normal subgroups of $G$ contained in $W_\rS$ are the elements of $\mathscr M_\rS^0$, which all have trivial intersection with $H$, while $M \leq H$ by definition. 
Thus $H \cap W'$ contains no minimal normal subgroups, hence is trivial. 

We have seen that $G = HW' = HN$ and $N \cap H = H \cap W' = W' \cap N = \{\ee\}$. 
We may thus invoke \Cref{lem:Goursat}, which provides an $H$-equivariant isomorphism $W' \to N$. 
It follows that $W'$ is a minimal normal subgroup of $G$, and that $\mathscr M= \{W', N\}$. 
Let $\sigma'$ and $\sigma''$ be the irreducible representations of $G$ obtained by composing the respective canonical projections $G \to G/W' \cong H$ and $G \to G/N \cong H$ with $\rho$. 
Note that by Frobenius reciprocity, both $\sigma'$ and $\sigma''$ are constituents of $\ind_H^G(\rho)$. 
Nonetheless, since $\dim(\ind_H^G(\rho)) = |N| \cdot \dim(\rho) > 2 \dim(\rho)$, the representation $\ind_H^G(\rho)$ must have an irreducible component $\sigma$ inequivalent to $\sigma'$ and $\sigma''$. 

It turns out that $\sigma'$ (resp.\ $\sigma''$) is the only irreducible constituent of $\ind_H^G(\rho)$ that is trivial on $W'$ (resp.\ $N$). 
Indeed, let $\tau$ be an irreducible constituent of $\ind_H^G(\rho)$ that is trivial on $W'$. 
Since $H$ maps (isomorphically) onto $G/W'$, we see that $\tau$ and its restriction $\restr{\tau}{H}$ have the same image; in particular, $\restr{\tau}{H}$ is irreducible. 
By Frobenius reciprocity, $\restr{\tau}{H}$ contains $\rho$, hence $\restr{\tau}{H} = \rho$. 
Therefore $\tau$ is the composite of the projection $G \to G/W' \cong H$ with $\rho$, in other words, $\tau = \sigma'$. 
The same argument yields the same conclusion for $\sigma''$ and $N$. 

It follows from the above that $\sigma$ is non-trivial on every element of $\mathscr M$. 
Since $\restr{\sigma}{H}$ contains $\rho$, which is faithful on $H$, we also know that $\sigma$ is non-trivial on every element of $\mathscr M'$. 
Altogether, $\sigma$ is non-trivial on every minimal normal subgroup of $G$, which ensures again that $\sigma$ is faithful. 
\end{proof}

Given a $\Z[G]$-module $A$ and submodules $C \leq B \leq A$, we say that the quotient $\Z[G]$-module $B/C$ is a \emphb{subquotient} of $A$ (sometimes also called a section of $A$). 

\begin{prop} \label{prop:inducefaithfultofaithful}
Let $H \leq G$ be finite groups, and suppose that $H$ has a faithful irreducible representation $\rho$. 
Let also $\mathscr M_\rA $ be the set of all abelian minimal normal subgroups $M$ of $G$ with $M \cap H =\{\ee\}$. 
Assume that for each non-empty subset $\mathscr W \subseteq \mathscr M_\rA$, there exists $V \in \mathscr W$ such that the $\Z[H]$-module $V$ is not isomorphic to any subquotient of the $\Z[H]$-module $\langle W \mid W \in \mathscr W \setminus \{V\} \rangle$. 
Then $G$ has a faithful irreducible representation $\sigma$ whose restriction to $H$ contains $\rho$. 
\end{prop}   
\begin{proof}
We proceed by induction on $d = |G|-|H|$. 
In the base case $d=0$, we have $G=H$ and the required assertion holds by taking $\sigma = \rho$. 

We assume henceforth that $d>0$. 
Let $\mathscr M$ be the set of all minimal normal subgroups $M$ of $G$ with $M \cap H =\{\ee\}$. 
Hence $\mathscr M_\rA \subseteq \mathscr M$. 

\medskip
We first consider the case where $\mathscr M_\rA \neq \mathscr M$, that is, where $G$ contains a non-abelian minimal normal subgroup $M$ intersecting $H$ trivially. 
Let $N \leq M$ be a minimal normal subgroup of $HM \cong H \ltimes M$. 
Note that $M$ is a direct product of copies of a non-abelian finite simple group $S$. 
Since $N$ is normal in $M$, it is a (non-trivial) product of some of these copies of $S$; in particular $N$ is non-abelian. 
We set $H' = H N \cong H \ltimes N$. 
By \Cref{lem:semidirect}, the group $H'$ has a faithful irreducible representation $\rho'$ whose restriction to $H$ contains $\rho$. 

As any subquotient of a $\Z[H']$-module $A$ is also a subquotient of $A$ viewed as $\Z[H]$-module, the induction hypothesis is satisfied by the subgroup $H' \leq G$ and its faithful irreducible representation $\rho'$, proving the required conclusion in this case.

\medskip
It remains to treat the case where $\mathscr M_\rA = \mathscr M$. 
Let $k =|\mathscr M_\rA |$. 
By the hypothesis made on $\mathscr M_\rA$, we may enumerate the elements of $\mathscr M_\rA$, say $M_1, \dots, M_k$, in such a way that for $i = 2, \dots, k$, the normal subgroup $M_{i}$ (viewed as a $\Z[H]$-module) is not isomorphic to a subquotient of the $\Z[H]$-module $W_i = \langle M_j \mid 1 \leq j < i\rangle$. 

We set $H_0 = H$, and for $i=1, \dots, k$, we define 
\[
H_i = H  \langle M_j \mid 1\leq j \leq i\rangle.
\]

\begin{claim} \label{claim:1}
    For each $i \in \{1, \dots, k\}$, we have $\langle M_j \mid 1\leq j \leq i\rangle \cong M_1 \times \dots \times M_i$.
\end{claim}
We prove the claim by induction on $i$, the result being trivial for $i=1$. 

Let $i> 1$. 
By the choice of the ordering, the normal subgroup $M_i$ is not isomorphic to a $\Z[H]$-submodule of $W_i = \langle M_j \mid 1 \leq j < i\rangle$. 
In particular, $M_i$ is not contained in $W_i$, hence $M_i \cap W_i = \{\ee\}$ since $M_i$ is a minimal normal subgroup of $G$. 
Thus $M_i W_i \cong M_i \times W_i$, and we have $W_i \cong M_1 \times \dots \times M_{i-1}$ by induction. 
This proves \Cref{claim:1}. 

\begin{claim} \label{claim:2}
    For each $i \in \{1, \dots, k\}$, we have $M_{i} \not \leq H_{i-1}$. 
\end{claim}
Suppose for the sake of contradiction that $M_i \leq H_{i-1} = HW_i$. 

Let $\pi \colon H_{i-1} \to H_{i-1}/W_i \cong H/H \cap W_{i-1}$ denote the canonical projection. Let $M'_i \leq H$ be the preimage of $\pi(M_i)$  in $H$. 
Thus $M_i \leq M'_i W_i $, from which we deduce that $M'_i \leq M_i W_i \cong M_i \times W_i$, in view of \Cref{claim:1}. 
By construction, the canonical projection from $M_i'$ to the first factor $M_i$ is surjective. 
Since $M_i \in \mathscr M_A$, we have in addition that $M_i \cap H = \{\ee\}$, hence $M_i \cap M'_i = \{\ee\}$. 
This shows that the canonical projection from $M_i'$ to the second factor $W_i$ is injective. 
Altogether, we see that $M_i$ is a quotient of the $\Z[H]$-submodule $M_i'$ of $W_i$, a contradiction. 

\smallskip
Now, set $\rho_0  =\rho$. 
For $i=1, \dots, k$, we inductively define an irreducible representation $\rho_i$ of $H_i$ as follows. 
The kernel of the induced representation $\ind_{H_{i-1}}^{H_i}(\rho_{i-1})$ is certainly contained in $H_{i-1}$. 
In particular, $\ind_{H_{i-1}}^{H_i}(\rho_{i-1})$ is non-trivial on $M_i$, since $M_i \not \leq H_{i-1}$ by \Cref{claim:2}. 
Therefore, the induced representation $\ind_{H_{i-1}}^{H_i}(\rho_{i-1})$ has an irreducible constituent $\rho_i$ whose restriction to $M_i$ is non-trivial. 

\begin{claim} \label{claim:3}
    The restriction of $\rho_k$ to $H$ contains $\rho$, and the restriction of $\rho_k$ to $M_i$ is non-trivial for all $i \in \{1, \dots, k\}$. 
\end{claim}
Indeed, the restriction of $\rho_k$ to $M_k$ is non-trivial by construction, and by Frobenius reciprocity, the restriction of $\rho_k$ to $H_{k-1}$ contains $\rho_{k-1}$. 
By induction, it follows that for every $i \in \{0, \dots, k\}$, the restriction of $\rho_k$ to $H_i$ contains $\rho_i$. 
In particular, the restriction of $\rho_k$ to $H_0 = H$ contains $\rho_0 = \rho$. 
Moreover, since $\rho_i$ is non-trivial on $M_i$, it follows that $\rho_k$ is non-trivial on $M_i$ for all $i \in \{1, \dots, k\}$. 

\smallskip
Finally, we take $\sigma$ to be any one of the irreducible components of $\ind_{H_k}^G(\rho_k)$. 
By Frobenius reciprocity, the restriction of $\sigma$ to $H_k$ contains $\rho_k$. 
It follows from \Cref{claim:3} that $\rho$ is contained in $\restr{\sigma}{H}$ and that $\sigma$ is non-trivial on every element of $\mathscr M_\rA = \mathscr M$. 
As $\rho$ is faithful on $H$, this implies that $\sigma$ is non-trivial on every minimal normal subgroup of $G$. 
Thus $\sigma$ is faithful, concluding the proof of \Cref{prop:inducefaithfultofaithful}. 
\end{proof}
    
The core of \Cref{thm:faithfulcenterpreservingirrep} is packaged in the following more technical theorem. 

\begin{thm} \label{thm:main-tech}
Let $H$ be a finite group possessing a faithful irreducible representation $\rho$. Let $p_1, \dots, p_m$ be distinct prime divisors of the order of $\rZ(H)$, and for each $i \in \{1, \dots, m\}$, let $h_i \in \rZ(H)$ be an element whose order is a power of $p_i$. 

Let $\iota \colon H \to G$ be an injective homomorphism to a finite group $G$, such that $\iota(h_i) \not \in \rZ(G)$ for any $i$. 
Then $G$ has an irreducible representation $\sigma$ such that $\rho$ is contained in $\sigma \circ \iota$ (in particular $\sigma \circ \iota$ is injective) and $\iota(h_i) \not \in \rZ(\sigma)$ for any $i \in \{1, \dots, m\}$. 
\end{thm}

\begin{proof}
Let $k \in \{0, 1, \dots, m\}$. 
Observe that the conclusion of \Cref{thm:main-tech} is equivalent to the following statement for $k = m$.

\begin{mainclaim*}
For each subset $I \subseteq \{1, \dots, m\}$ of size $k$ and for each finite group $\mathscr G$ in which $H$ embeds as a subgroup in such a way that $h_i \not \in \rZ(\mathscr G)$ for any $i \in I$, there exists an irreducible representation $\sigma$ of $\mathscr G$ whose restriction to $H$ contains 
$\rho$, and such that $h_i \not \in \rZ(\sigma)$ for any $i \in  I$.
\end{mainclaim*}

We will prove the Main Claim by induction on $k = |I|$. 
For the base case $k=0$, we induce the given representation $\rho$ of $H$ up to $\mathscr G$ and define $\sigma$ as one of the irreducible components of the latter. 
By Frobenius reciprocity, the restriction of $\sigma$ to $H$ contains $\rho$. 
Thus the Main Claim holds for $k=0$. 

Let now $k>0$, and assume that the Main Claim holds for $k-1$. 
We fix a subset $I \subseteq \{1, \dots, m\}$ of size $k$. 
We shall prove that the Main Claim holds for that subset $I$, by induction on the order of $\mathscr G$. 

In the base case, we have $\mathscr G = H$, hence the result holds by setting $\sigma = \rho$. 
We assume henceforth that $H$ is properly contained in $\mathscr G$. 
\smallskip

For each $i \in I$, we know by assumption that $h_i$ is not central in $\mathscr G$. 
Hence the commutator subgroup $[h_i, \mathscr G]$ is non-trivial. 
Let $K_i = \llag [h_i, \mathscr G] \rrag$ be its normal closure in $\mathscr G$, so that $K_i$ is a non-trivial normal subgroup of $\mathscr G$.
It will serve to witness whether $h_i$ becomes central: the image of $h_i$ under a homomorphism $\varphi$ is central in $\varphi(\mathscr{G})$ if and only if $\varphi(K_i)$ is trivial. 
For each $i \in I$, we choose a minimal normal subgroup $M_i$ of $\mathscr G$ contained in $K_i$. 

\begin{step} \label{step:main-tech1}
{If there exists $i \in I$ such that $M_i \cap H \neq \{\ee\}$, then the conclusion holds.}
\end{step}

We apply the induction hypothesis on $k$ using the subset $I' = I \setminus \{i\}$. 
Denote by $\sigma$ the irreducible representation  of $\mathscr G$  afforded thereby. 
Thus $\rho$ is contained in $\restr{\sigma}{H}$, and $\sigma(h_j)$ is not central in $\sigma(\mathscr G)$ for any $j \neq i$. 
Furthermore, the restriction of $\sigma$ to $H$ is injective, hence the restriction of $\sigma$ to  $M_i \cap H$ is also injective. 
In particular, $\sigma(M_i)$ is non-trivial, hence $\sigma(h_i)$ is not central in $\sigma(\mathscr G)$. 
\smallskip

The Main Claim therefore holds in this case, and we assume henceforth that $M_i \cap H = \{\ee\}$ for all $i \in I$. 

\begin{step} \label{step:main-tech2}
    If there are distinct $i, j \in I$ such that $K_j \subseteq K_i$, then the conclusion holds.
\end{step}

We again apply the induction hypothesis on $k$ using the subset $I' = I \setminus \{i\}$. 
Denote by $\sigma$ the irreducible representation  of $\mathscr G$ afforded thereby. 
Since $\sigma(h_j)$ is not central in $\sigma(\mathscr G)$ for $j \neq i$, the images $\sigma(K_j)$ are non-trivial. 
Hence also $\sigma(K_i)$ is non-trivial, so that $\sigma(h_i) \not \in \rZ(\sigma)$. 
\smallskip

From hereon, we assume additionally that $K_j \subseteq K_i$ only when $i=j$. 

\begin{step} \label{step:main-tech3}
{If there exists $i \in I$ such that $M_i \neq K_i$, then the  conclusion holds.}
\end{step}
 
In view of \Cref{step:main-tech2}, we must have $K_j \neq M_i$ for all $j \in  I \setminus \{i\}$. 
By the hypothesis of \Cref{step:main-tech3}, we also have $K_i \neq M_i$. 
Since $M_i$ is a minimal normal subgroup of $\mathscr G$, it follows that for each $j \in I$, either $M_i \cap K_j$ is trivial or $M_i$ is properly contained in $K_j$. 
In all cases, we see that the image of $K_j$ in the quotient $\mathscr G/M_i$ is non-trivial. 
Hence  the image of $h_j$ is non-central in that quotient for all $j \in I$. 
Moreover, $H$ maps injectively to $\mathscr G/M_i$ by the assumption made after \Cref{step:main-tech1}. 
Therefore, the induction hypothesis on the order of $\mathscr G$ affords the conclusion. 

\begin{step} \label{step:main-tech4}
If $M_i = K_i$ for all $i \in I$, then the conclusion holds. 
\end{step}

In other words, we assume that $K_i$ is a minimal normal subgroup of $\mathscr G$ for all $i \in I$. 
By the assumption made after \Cref{step:main-tech2}, we have that $M_i \neq M_j$ for all $i \neq j$. 
In particular, $M_i \cap M_j$ is trivial for all $i \neq j$, hence the normal subgroups $M_i$ commute pairwise. 

Let $\mathscr M = \{M_i \mid i \in I\}$ and $\mathscr M'$ be the set of those minimal normal subgroups of $\mathscr G$ that have a non-trivial intersection with $H$. 
Suppose that $\mathscr G$ has a minimal normal subgroup $N$ not contained in $\mathscr M \cup \mathscr M'$. Then $H$ and each $M_i$ maps injectively to the quotient $\mathscr G/N$. 
In that case, the required conclusion holds by induction on the order of $\mathscr G$. 
We will thus assume that the set of all minimal normal subgroups of $\mathscr G$ coincides with $\mathscr M \cup \mathscr M'$.

By the standing assumption following \Cref{step:main-tech1}, we know that $\mathscr M$ and $\mathscr M'$ are disjoint. Let $\mathscr M_\rA$ be the set consisting of the abelian elements of $\mathscr M$, and $\mathscr M'_\rA$ of those of $\mathscr M'$. 
We view again each $N \in \mathscr M_\rA \cup \mathscr M'_\rA$ as a $\Z[\mathscr G]$-module, with $\mathscr G$ acting by conjugation. 

\begin{claim*}
{For each non-empty subset $\mathscr W \subseteq \mathscr M_\rA$, there exists $V \in \mathscr W$ such that the $\Z[H]$-module $V$ is not isomorphic to any subquotient of the $\Z[H]$-module $\langle W \mid W \in \mathscr W \setminus \{V\}\rangle$.}
\end{claim*}

Let $i, j \in I$ be such that $M_i, M_j \in \mathscr M_\rA$. 
Since $M_j$ is normal in $\mathscr G$, we have $[h_i, M_j] \subseteq M_j$. 
Since $M_i = K_i$, we also have $[h_i, M_j] \subseteq [h_i, \mathscr G] \subseteq M_i$. As $M_i \cap M_j = \{\ee\}$, we infer that if $i \neq j$ then $[h_i, M_j] = \{\ee\}$, so that  $h_i$ commutes with $M_j$. 

Let now $\mathscr W \subseteq \mathscr M_\rA$ be non-empty, so that $\mathscr W = \{M_j \mid j \in J\}$ for some non-empty subset $J \subseteq I$. 
Suppose for a contradiction that for each $j \in J$, the $\Z[H]$-submodule $M_j$ is isomorphic to a subquotient of the $\Z[H]$-module $W_j = \langle M_i \mid i \in J \setminus \{j\} \rangle$. 
We have seen above that $h_j$ commutes with $M_i$ for all $i \neq j$, hence $h_j$ commutes with $W_j$. 
So $h_j$ acts trivially on every subquotient of $W_j$. 
In particular, $h_j$ commutes with $M_j$. 

Next, we deduce that $M_j$ must be an elementary abelian $p_j$-group for all $j \in J$. 
Indeed, pick $g \in \mathscr G$ such that $[h_j, g]$ is non-trivial (such an element exists because by assumption, $h_j$ is not central). 
In view of \Cref{step:main-tech3}, we have $[h_j, g] \in M_j$. 
Using that $h_j$ commutes with $M_j$, it follows that the commutator map $h \mapsto [h, g]$ is a non-trivial homomorphism from $\langle h_j \rangle $ to $M_j$. 
In consequence, $M_j$ contains an element of order $p_j$. 
Since $M_j$ is an abelian minimal normal subgroup of $\mathscr G$, it must be an elementary abelian $p_j$-group. 

But then $M_j$ is a subquotient of $W_j$ whose order is prime to that of $W_j$; this is absurd. 
The Claim stands proven. 

\medskip 
The Claim ensures that the hypothesis of \Cref{prop:inducefaithfultofaithful} is satisfied. 
Thus $\mathscr G$ has a faithful irreducible representation $\sigma$ whose restriction to $H$ contains $\rho$. 
Since $\sigma$ is faithful, we have $\rZ(\sigma) = \rZ(G)$, hence that representation obviously satisfies the requirements of the Main Claim, concluding \Cref{step:main-tech4}, and with it, the proof of \Cref{thm:main-tech}. 
\end{proof}

We can now complete the proof of \Cref{thm:faithfulcenterpreservingirrep}. 

\begin{proof}[Proof of \Cref{thm:faithfulcenterpreservingirrep}]
Let $p_1, \dots, p_n$ be the distinct prime divisors of the center $\rZ(H)$ of $H$. 
For each $i \in \{1, \dots, n\}$, let $P_i$ be the $p_i$-Sylow subgroup of $\rZ(H)$. 
By hypothesis, the subgroup $H$ has a faithful irreducible representation. 
Hence the center of $H$ is cyclic, and so is the group $P_i$ for each $i \in \{1, \dots, n\}$. 
Upon re-indexing the $P_i$'s, we may assume that the subset $\{ i \in \{1, \dots, n\} \mid P_i \not \leq \rZ(G)\}$  is $\{1, \dots, m\}$ for some $m \leq n$. 
Thus $P_j \leq \rZ(G)$ for all $j >m$. 
For all $i \in \{1, \dots, m\}$, we choose an element $h_i \in P_i$ that is a generator of the smallest subgroup of $P_i$ not contained in $\rZ(G)$. 

With these choices, let $\sigma$ be an irreducible representation afforded by \Cref{thm:main-tech}. 
The restriction of $\sigma$ to $H$ contains $\rho$, hence $\sigma$ is an irreducible component of $\ind_H^G(\rho)$ by Frobenius reciprocity. 
Since $\rho$ is faithful, we infer that  $\rZ(\sigma) \cap H \leq \rZ(H)$. 
Since $\rZ(H)$ is cyclic, so is the subgroup $\rZ(\sigma) \cap H$. 
For all $i \leq m$ we have $h_i \not \in \rZ(\sigma)$ and $h_i^{p_i} \in \rZ(G) \leq \rZ(\sigma)$, whereas for $j >m$ we have $P_j \leq \rZ(G) \leq \rZ(\sigma)$. 
It follows that the Sylow subgroups of $H \cap \rZ(\sigma)$ are $\langle h_1^{p_1}\rangle, \ldots, \langle h_m^{p_m}\rangle$ and $P_{m+1}, \ldots, P_n$. 
In other words, every Sylow subgroup of $H \cap \rZ(\sigma)$ is contained in $\rZ(G)$. Hence $H \cap \rZ(\sigma) \leq \rZ(G)$, as required. 
\end{proof}

In passing, we record the following straightforward consequence of \Cref{thm:faithfulcenterpreservingirrep}. 
Let $\rZ_2(G)$ denote the second center of $G$, that is, the preimage in $G$ of the center of $G / \rZ(G)$. 
The necessity of the additional assumption on $H \cap \rZ_2(G)$ below will be shown in \Cref{exe:D4C4}. 

\begin{cor} \label{cor:centerpreservingirrep}
Let $H \leq G$, and suppose that $\rZ_2(G) \cap H = \rZ(G) \cap H$. 
If $H$ admits an irreducible representation $\rho$ whose kernel $K$ is contained in $\rZ(G)$, then $\ind_H^G(\rho)$ has an irreducible constituent $\sigma$ whose restriction to $H$ has kernel $K$, and which is center-preserving on $H$. 
\end{cor}
\begin{proof}
By assumption, the representation $\rho$ factors through a faithful irreducible representation $\rho'$ of the quotient $H/K$. 
Let $\sigma$ be the inflation to $G$ of the irreducible representation $\sigma'$ of $G/K$ afforded by \Cref{thm:faithfulcenterpreservingirrep}, applied to $\rho'$ and the pair $H / K \leq G / K$. 
By construction, $\rZ(\sigma) \cap H$ is the preimage of $\rZ(\sigma') \cap H/K = \rZ(G/K) \cap H/K$ under the quotient map $G \to G/K$. 
Because $K$ is central, the preimage of $\rZ(G/K)$ is contained in $\rZ_2(G)$. 
Thus $\rZ(\sigma) \cap H$ is contained in $\rZ_2(G) \cap H = \rZ(G) \cap H$. 
This shows that $\sigma$ is center-preserving on $H$. 

Lastly, the restriction of $\sigma'$ to $H/K$ contains $\rho'$ by Frobenius reciprocity, and since $\sigma$ and $\rho$ are inflated from $\sigma'$ and $\rho'$, it follows that the restriction of $\sigma$ to $H$ contains $\rho$. 
This implies that $\sigma$ is a constituent of $\ind_H^G(\rho)$, and that $\ker(\sigma) \cap H \leq \ker(\rho) = K$. 
\end{proof}

\subsection{Examples} \label{subsec:examplescpr}
For the duration of this section, we assume for simplicity that $\Fi$ is algebraically closed (always of cross-characteristic). 

\begin{exe} \label{exe:exactly-one}
Consider indeed the Heisenberg group $\mathscr H$ over $\Z/4\Z$, that is, 
\[
\mathscr H = \langle x, y, z \mid x^{4}, y^{4}, z^{4}, [x, z], [y, z], [x,y]z^{-1}\rangle. 
\]
Let $G = \langle x, y^2 \rangle$ be the subgroup of $\mathscr H$ of order $16$ generated by $x$ and $y^2$, and set $H = \langle x \rangle \cong C_4$. 
The center of $G$ is $\langle x^2, z^2 \rangle \cong C_2 \times C_2$, so that the kernel of each irreducible representation of $G$ contains $x^2$ or $z^2$ or $x^2 z^2$. 
An irreducible representation $\sigma$ of $G$ is faithful on $H$ if and only if $\sigma(x^2)$ is non-trivial, while it is center-preserving on $H$ if and only if $\sigma(z^2)$ is non-trivial. 
Thus, in order for $\sigma$ to satisfy both conditions, one must have that $x^2 z^2 \in \ker (\sigma)$. 

The group $H$ has exactly two faithful irreducible characters. 
If $\rho$ is any of them, then the induced representation $\ind_H^G(\rho)$ decomposes as a direct sum $\sigma_0 \oplus \sigma_1 \oplus \sigma_2$ of three irreducible representations of $G$, with $\sigma_0$ of degree~$2$ and $\sigma_1, \sigma_2$ of degree~$1$. 
In accordance with \Cref{thm:faithfulcenterpreservingirrep}, one can check that the representation $\sigma_0$ is the unique irreducible representation of $G$ that is faithful and center-preserving on $H$. 
\end{exe}

\begin{rem} \label{rem:centerpreservingnormalsubgroup}
\Cref{exe:exactly-one} shows that in \Cref{thm:faithfulcenterpreservingirrep}, one can generally not expect more than one irreducible constituent of $\ind_H^G(\rho)$ to be center-preserving on $H$. 
There is however a special case where \emph{every} irreducible constituent of $\ind_H^G(\rho)$ is center-preserving on $H$: when $H$ is a normal subgroup of $G$. 
By virtue of Frobenius reciprocity, this is a consequence of the following basic algebraic fact. 

\smallskip
\textit{Let $\sigma \colon G \to L$ be a homomorphism and $H \triangleleft G$ be a normal subgroup. 
If the restriction of $\sigma$ to $H$ is injective, then $\sigma$ is center-preserving on $H$. }
\smallskip

Indeed, an element $h \in H$ maps to a central element of $\sigma(G)$ if and only if $[g, h] \in \ker(\sigma)$ for all $g \in G$. 
Since $H$ is normal, we also have $[g, h] \in H$ for all $g \in G$. Because $\restr{\sigma}{H}$ is injective, we infer that $[g, h] \in H \cap \ker(\sigma)$ is trivial, hence $H \cap \rZ(\sigma) \leq \rZ(G)$. 
\end{rem}

\begin{rem} \label{rem:hypothesesnecessity}
The next two examples show that if one replaces the hypothesis \guillemotleft{$H$ has a faithful irreducible representation $\rho$}\guillemotright{} in \Cref{thm:faithfulcenterpreservingirrep} by any of the weaker assumptions 
\begin{enumerate}[leftmargin=2em,itemsep=.2ex,label=\textup{$\bullet$}]
    \item \guillemotleft{$G$ has an irreducible representation whose restriction $\rho$ to $H$ is faithful}\guillemotright{}, or 
    \item \guillemotleft{$H$ has an irreducible representation $\rho$ whose induction to $G$ is faithful}\guillemotright{} (equivalently, such that $\mathrm{core}_G(\ker(\rho)) = \{\ee\}$), or
    \item \guillemotleft{$H$ has an irreducible representation $\rho$ such that $\ker(\rho) \leq \rZ(G)$}\guillemotright{}, 
\end{enumerate}
it may happen that no irreducible representation of $G$ is center-preserving on the subgroup $H$, let alone an irreducible constituent of $\ind_H^G(\rho)$. 
In particular, the conclusion of \Cref{thm:main-tech} does not hold under any of these weaker assumptions, and neither does the potential variant of \Cref{cor:faithfulcenterpreservingirrep} where (i) would be replaced by one of the weaker conditions above, and (ii) by the existence of an irreducible representation of $G$ which is center-preserving on $H$. 

In \Cref{cor:centerpreservingirrep}, we provided an intermediate condition that does ensure the existence of an irreducible representation of $G$ that is center-preserving (but not necessarily faithful) on $H$. 
\end{rem}

\begin{exe} \label{exe:D4D4D4}
Let $D_4$ be the dihedral group of order $8$, and $E_1, E_2, E_3$ be three groups isomorphic to $D_4$. 
Let $z_i \in E_i$ be the non-trivial element of the center of $E_i$, and set 
\[
G = (E_1 \times E_2 \times E_3) / \langle z_1 z_2 z_3\rangle. 
\]
The group $G$ appeared in \cite[Appendix]{CH20} as an example of a group whose irreducible representations each have a non-abelian kernel. 
Denote by $y \mapsto \bar y$ the canonical projection of $E_1 \times E_2 \times E_3$ to $G$. 
Let $x_i \in E_i$ be a non-central element of order~$2$, and let $H$ be the image of $\langle \bar x_1, \bar x_2, \bar x_3 \rangle$ in $G$. 
Observe that $\rZ(G) \cap H = \{\ee\}$. 
Moreover, $G$ has irreducible representations whose restrictions to $H$ are faithful. 
However, for each such representation, the image of some $\bar{x}_i$ is central. 
Indeed, the center of $G$ is isomorphic to $C_2 \times C_2$, and its non-trivial elements are $\bar{z}_1, \bar{z}_2$ and $\bar{z}_3$. 
For each irreducible representation $\sigma$ of $G$, there must exist some $i$ such that $\sigma(\bar{z}_i)$ is trivial; hence $\sigma(\bar{x}_i)$ is central. 
Of course, the group $H$ itself has no faithful irreducible representation. 
However, $H$ does have an irreducible representation $\rho$ with $\core_G(\ker(\rho)) = \{\ee\}$.
\end{exe}

\begin{exe} \label{exe:D4C4}
Let
\[
G = D_4 \ltimes C_4 = \langle a,b,c \mid a^4, b^2, c^4, b^{-1}aba, a^{-1}cac ,b^{-1}cbc  \rangle,
\]
where the standard generators $a$, $b$ of $D_4$ act on $C_4$ by inverting its generator $c$. 
Note that $\rZ(G) = \langle a^2, c^2 \rangle \cong C_2 \times C_2$. 
The quotient $G/\langle a^{2}c^{2} \rangle$ is isomorphic to the central product 
\(
D_4 \circ C_4
\)
(also known as the Pauli group), that is, to the quotient of $D_4 \times C_4$ by $\langle a^{2} c^{2} \rangle$. 
In fact, if $x \mapsto \bar{x}$ denotes the quotient map and $d = abc$, then $G/\langle a^{2}c^{2} \rangle = \langle \bar{a}, \bar{b} \rangle \circ \langle \bar{d} \rangle$ with $\langle \bar{a}, \bar{b} \rangle \cong D_4$ and $\langle \bar{d} \rangle \cong C_4$. 
The group $D_4 \circ C_4$ has two faithful irreducible representations of degree $2$, whose inflations $\rho_1$ and $\rho_2$ to $G$ are irreducible representations with central kernel.  
However, as $\rZ(G/\langle a^{2}c^{2} \rangle) = \langle \bar{d} \rangle$, we have $\rZ(\rho_1) = \rZ(\rho_2) = \langle a^2, d \rangle \gneq \rZ(G)$. 
Every other irreducible representation of $G$ is either of degree 1, or inflated from a quotient isomorphic to $D_4$. 
In particular, the center of these irreducible representations has index at most $4$. 
Altogether, $G$ has no irreducible representation which is center-preserving. 
The group $G$ is of smallest size among all such examples (but there are five others of order $32$). 
\end{exe}

\section{Faithful projective representations} \label{sec:fpr}

The goal of this section is to interpret the notion of center-preserving representation in the context of projective representations. 
After some recollections on projective representations and their relationship with central extensions by $\Fi^\times$, we first deal with a single group $G$, and address the corresponding questions relative to a subgroup $H \leq G$ afterwards.

\subsection{Projective representations} \label{subsec:projectiverepresentations}

Recall that a \emphb{projective $\Fi$-representation of $G$} is a homomorphism $\ol{\rho} \colon G \to \PGL(V)$ to the projective linear group of a vector $\Fi$-space $V$. 
The field $\Fi$ is fixed and will be omitted from the notation, as we did for linear representations. 
With any projective representation $\ol{\rho}$ of $G$ is associated a central extension
\[
1 \to \Fi^\times \to \tilde{G}_\rho \to G \to 1
\]
and a linear representation ${\rho}$ of $\tilde{G}_\rho$ lifting $\ol{\rho}$, in a way that the cohomology class $c_{\ol{\rho}} \in \rH^2(G,\Fi^\times)$ of the extension $\tilde{G}_\rho$ is the image of $\ol{\rho}$ under the map $\rH^1(G,\PGL(V)) \to \rH^2(G,\Fi^\times)$ induced in group cohomology by the exact sequence
\[
1 \to \Fi^\times \to \GL(V) \to \PGL(V) \to 1
\]
of groups with trivial action of $G$. 
The image $c_{\ol{\rho}}$ (or just $c$ if $\ol{\rho}$ is fixed) of $\ol{\rho}$ in $\rH^2(G,\Fi^\times)$ is called the \emphb{class of $\ol{\rho}$}. 
(The class of $\ol{\rho}$ is also the image of the class of the central extension $\GL(V) \to \PGL(V)$ under the map $\rH^2(\PGL(V),\Fi^\times) \to \rH^2(G,\Fi^\times)$ induced by $\ol{\rho}$.)
The set of (equivalence classes of) projective representations of $G$ is partitioned along their classes, and we will call \emphb{$c$-projective} any projective representation of class $c$.\footnote{The term `$c$-projective representation' is sometimes abbreviated `$c$-representation' in the literature. } 

Schur showed last century that there is an equivalence between the categories of projective representations of $G$ with class $c \in \rH^2(G,\Fi^\times)$, and of linear representations of the central extension $\tilde{G}_c$ determined by (any $2$-cocycle representing) $c$ under which the central kernel $\Fi^\times$ acts by scalars. 
When $\Fi$ is algebraically closed, he also showed that every finite group admits what is now called a \emphb{Schur cover}, a central extension of $G$ by the finite abelian group $\rH_2(G,\Z)$ to which every projective $\Fi$-representation of $G$ lifts. 
When $\Fi$ is not algebraically closed, the theory is more complicated. 
Indeed, already the smallest non-trivial group $C_2$ admits a projective $\Q$-representation that does not lift to any finite central extension of $C_2$: 

\begin{exe} \label{exe:liftcyclicgroup}
Let $a \in \Fi^\times$. 
The projective $\Fi$-representation of the cyclic group $C_n$ of order $n$ given by 
\[
C_n = \langle h \rangle \to \PGL_n(\Fi): h \mapsto {\scriptstyle {
\left[ {\begin{matrix}  0 & 1 & 0 & \cdots & 0 \\ 
                            0 & \ddots & \ddots & \ddots & \vdots \\ 
                            \vdots & \ddots & \ddots & \ddots & 0 \\
                            0 & \ddots & \ddots & \ddots & 1 \\
                            a & 0 & \cdots & 0 & 0 
\end{matrix}} \right]}}
\]
admits a finite cover in $\GL_n(\Fi)$ if and only if $a$ is the product in $\Fi$ of an $n$\textsuperscript{th} power and a root of unity. 
\end{exe}

\medbreak
For our purposes, it will be convenient to have at hand a subextension of $\tilde{G}_c$ which is a finite group. 
To this end, let $\muF$ denote the group of roots of unity in $\Fi$, and obverve that the inclusion $\muF \leq \Fi^\times$ induces an embedding of $\rH^2(G,\muF)$ in $\rH^2(G,\Fi^\times)$. 
Indeed, the induced map fits in the long exact sequence
\[
\rH^1(G,\Fi^\times / \muF) \to \rH^2(G,\muF) \to \rH^2(G,\Fi^\times) \to \rH^2(G,\Fi^\times / \muF),
\]
and because $\Fi^\times / \muF$ is torsion-free, the group $\rH^1(G,\Fi^\times / \muF) = \Hom(G,\Fi^\times / \muF)$ is trivial. 
For the remainder of this section, we will identify $\rH^2(G,\muF)$ with its image in $\rH^2(G,\Fi^\times)$. 
By definition, this image coincides with the image in cohomology of the subgroup $\rZ^2(G,\muF)$ of $\rZ^2(G,\Fi^\times)$;\footnote{The reader is entrusted not to confuse groups of $2$-cocycles with the various centers also appearing in the text under the letter $\rZ$.} 
or in other words, a class $c \in \rH^2(G,\Fi^\times)$ belongs to $\rH^2(G,\muF)$ if and only if $c$ can be represented by a $2$-cocycle whose values are all roots of unity in $\Fi$. 
Given $\zz \in \rZ^2(G,\muF)$, we let $\muz$ denote the finite subgroup of $\Fi^\times$ generated by the image of $\zz$. 

Note that $\rH^2(G,\muF)$ is \emph{not} the torsion subgroup of $\rH^2(G,\Fi^\times)$. 
The two may be distinct, and in fact the (possibly infinite) abelian group $\rH^2(G,\Fi^\times)$ is already torsion, of exponent dividing $|G|$. 
Yamazaki \cite[Proposition~3]{Yamazaki64a} proved that $\rH^2(G,\muF) = \rH^2(G,\Fi^\times)$ if and only if $\Fi^\times = \muF \cdot \Fi^{\times |G^\mathrm{ab}|}$ (cf.\ \Cref{exe:liftcyclicgroup}). 
This holds in particular when $\Fi$ is a finite field, when $\Fi = \R$, or when $\Fi^\times$ is a $|G^\mathrm{ab}|$-divisible group (that is, when every element of $\Fi^\times$ admits a $|G^\mathrm{ab}|$\textsuperscript{th} root), and a fortiori when $\Fi$ is algebraically closed. 

By definition of $\mu_\zz$, any $\zz \in \rZ^2(G,\muF)$ determines a central extension $G_\zz$ of $G$ by the finite group $\muz$, whose class in $\rH^2(G,\muz)$ is the image of the corestriction of $\zz$. 
We thus have the following embedding of central extensions of $G$: 
\[\begin{tikzcd}
	1 & {\Fi^\times} & {\tilde{G}_c} & G & 1 \\
	1 & {\mu_\zz} & {G_\zz} & G & 1
	\arrow[from=1-1, to=1-2]
	\arrow[from=1-2, to=1-3]
	\arrow[from=1-3, to=1-4]
	\arrow[from=1-4, to=1-5]
	\arrow[from=2-1, to=2-2]
	\arrow[hook, from=2-2, to=1-2]
	\arrow[from=2-2, to=2-3]
	\arrow[hook, from=2-3, to=1-3]
	\arrow[from=2-3, to=2-4]
	\arrow[equal, from=2-4, to=1-4]
	\arrow[from=2-4, to=2-5]
\end{tikzcd}\]
where $c$ is the image of $\zz$ in $\rH^2(G,\muF) \leq \rH^2(G,\Fi^\times)$. 
It follows that for $c \in \rH^2(G,\muF)$ and for any $\zz \in \rZ^2(G,\muF)$ representing $c$, every $c$-projective representation of $G$ lifts to a linear representation of the finite group $G_\zz$ (simply restrict to $G_\zz$ its lift to $\tilde{G}_c$). 
Conversely, if there exists a finite central extension of $G$ lifting even just one of the $c$-projective representations of $G$, then $c \in \rH^2(G,\muF)$. 
This equivalence can be found in \cite[Proposition~1]{Yamazaki64a}. 

Furthermore, Yamazaki \cite[Theorem~1]{Yamazaki64a} showed that every finite group $G$ admits a \emphb{Yamazaki cover}, that is, a finite central extension to which lifts any projective representation of $G$ whose class lies in $\rH^2(G,\muF)$. 
Thus, when say $\Fi^\times$ is a $|G^\mathrm{ab}|$-divisible group, every projective representation lifts to any Yamazaki cover of $G$. 
Yamazaki also gave a criterion \cite[Theorem~2]{Yamazaki64a} for the existence of a Yamazaki cover of smallest possible order $|\rH^2(G,\muF)| \cdot |G|$.\footnote{Yamazaki calls representations whose class belongs to $\rH^2(G,\muF)$ `linearizable', and calls Yamazaki covers of $G$ of order $|\rH^2(G,\muF)| \cdot |G|$ `quasi-representation-groups of $G$ over $\Fi$' (in line with the term `Darstellungsgruppe' used by Schur for the cover he introduced). } 
If now $\Fi$ is algebraically closed, one recovers from the above that any minimal Yamazaki cover is a Schur cover. 
As with Schur covers, a minimal Yamazaki cover need not be unique. 

\medbreak
A more modern approach relates the $c$-projective representations of $G$ with the modules over the $c$-twisted group algebra $\Fi_c[G]$ of $G$. 
When the characteristic of $\Fi$ does not divide $|G|$ (as is assumed throughout), $\Fi_c[G]$ is a semisimple $\Fi$-algebra with basis $G$, which is the quotient of the group algebra $\Fi[\tilde{G}_c]$ after identifying the images of $\Fi^\times$ in the coefficients $\Fi$ and in the group basis $\tilde{G}_c$. 
(We will not make use of twisted group algebras.) 

For additional background on projective representations, we refer the reader to \cite[VI]{BKZ18} and \cite{Karpilovsky85}. 
For the development of the theory over an arbitrary field, the reader is invited to consult the original works \cite{Yamazaki64a,Yamazaki64b} of Yamazaki. 
A description of minimal Yamazaki covers of $G$ when $\rH^2(G,\muF) = \rH^2(G,\Fi^\times)$ can be found in \cite[Theorem~4.8]{MargolisSchnabel20}. 
A criterion for the existence of a faithful projective representation can be found in \cite{Pahlings68} and \cite{BH13}. 
The remainder of this section will be concerned with projective representations that are faithful on a subgroup. 

\medbreak
To conclude this section, we record for later use the following well-known lemma. 

\begin{lem} \label{lem:orderofcocyles}
Let $A$ be an abelian group endowed with an action of $G$, and assume that $A$ is $|G|$-divisible. 
Each class $c \in \rH^2(G,A)$ admits a representative $\zz$ in $\rZ^2(G,A)$ of order dividing $|G|$. 
\end{lem}
\begin{proof}
Pick any $\zz \in \rZ^2(G,A)$ representing the class $c$, and set $f(g) = \prod_{h \in G} \zz(g,h)$. 
By multiplying $2$-cocycle identities, we compute
\[
f(g) \leftindex{^g} f(h) = \prod_{k \in G} \zz(g,hk) \prod_{k \in G} \leftindex{^g} \zz(h,k) = \prod_{k \in G} \zz(gh,k) \prod_{k \in G} \zz(g,h) = f(gh) \zz(g,h)^{|G|},
\]
and deduce that $\zz^{|G|}$ is the coboundary of $f \colon G \to A$.\footnote{This classical trick shows by the way that the exponent of $\rH^n(G,A)$ divides $|G|$ (regardless of whether $A$ is $|G|$-divisible).} 
Because $A$ is $|G|$-divisible, we can pick for each $g \in G$ an element $\dot{f}(g) \in A$ such that $\dot{f}(g)^{|G|} = f(g)$. 
The $2$-cocycle ${\zz}'$ given by
\[
{\zz}'(g,h) = \zz(g,h) \dot{f}(gh) \dot{f}(g)^{-1} \leftindex{^g} \dot{f}(h)^{-1}
\]
then represents the same class as $\zz$, and satisfies
\[
{\zz}'(g,h)^{|G|} = \zz(g,h)^{|G|} f(gh) f(g)^{-1} \leftindex{^g} f(h)^{-1} = 1
\]
on the nose. 
\end{proof}

\subsection{The absolute case} \label{subsec:absolutefpr}

\begin{defn}
Recall that a $c$-projective representation $\ol \sigma$ is \emphb{faithful} if $\ker(\ol \sigma) = \{\ee\}$. 
Some groups admit no faithful projective representation, so it is natural to introduce the following more flexible notion. 
We say a $c$-projective representation $\ol \sigma$ is \emphb{$c$-faithful} if $ \ker(\ol \sigma)$ coincides with the intersection $\rK_c(G)$ of the kernels of all irreducible $c$-projective representations of $G$. 
\end{defn}

Given a representation $\sigma$ of a group $G$, the \emphb{quasikernel} $\qker(\sigma)$ (sometimes \emphb{projective kernel}) \emphb{of $\sigma$} is the kernel of the projectivization of $\sigma$. 
Equivalently, $\qker(\sigma)$ is the preimage under $\sigma$ of the subgroup of transformations in $\sigma(G)$ that are scalars lying in $\Fi^\times$. 
The quasikernel $\qker(\sigma)$ is obviously a subgroup of $\rZ(\sigma)$. 
They coincide when $\sigma$ is irreducible and $\Fi$ is $|G|$-closed by virtue of Schur's lemma, but not in general. 
Here and throughout, we say for short that $\Fi$ is \emphb{$m$-closed} when $\Fi$ contains all $m$\textsuperscript{th} roots of unity. 
By definition, $\rK_1(G)$ is the intersection of the quasikernels of all irreducible representations of $G$. 

Given a class $c \in \rH^2(G, \Fi^\times)$, we denote by $\rZ_c(G)$ the image of the center $\rZ(\tilde{G}_c)$ under the quotient map $\tilde{G}_c \to G$ associated with the central extension
\(
1 \to \Fi^\times \to \tilde{G}_c \to G \to 1
\)
determined by $c$. 
If $c$ happens to lie in $\rH^2(G,\muF)$, so that $c$ can be represented by some $\zz \in \rZ^2(G,\muF)$, then the normal subgroup $\rZ_c(G)$ is also the image of the center $\rZ(G_\zz)$ of the subextension $G_\zz$ previously introduced. 
In case $c=1$ is the trivial class, we have $\rZ_1(G) = \rZ(G)$, but in general the obvious inclusion $\rZ_c(G) \leq \rZ(G)$ can be strict. 

Note that $\rZ_c(G)$ and $\rK_c(G)$ depend on $\Fi$ through $c$. 
When $\Fi^\times$ is $|G|$-divisible, the subgroup $\rZ_c(G)$ has the following representation-theoretic interpretation, which is easily proved using \Cref{thm:faithfulcenterpreservingirrep}. 

\begin{lem} \label{lem:c-center}
Let $c \in \rH^2(G,\muF)$. 
The normal subgroup $\rZ_c(G)$ contains the intersection $\rK_c(G)$ of the kernels of all irreducible $c$-projective representations of $G$. 
If $\muF$ is $|G|$-divisible, then $\rZ_c(G)$ coincides with $\rK_c(G)$. 
\end{lem}

Recall that $\rH^2(G,\muF) = \rH^2(G,\Fi^\times)$ when $\Fi^\times$ is $|G|$-divisible (see \cite[Proposition~3]{Yamazaki64a}). 
Thus, under this hypothesis, we see for any class $c \in \rH^2(G,\Fi^\times)$, that a $c$-projective representation $\ol \sigma$ is $c$-faithful if and only if $\ker(\ol \sigma) = \rZ_c(G)$. 

\begin{proof}
Let $\zz$ be a representative of $c$ in $\rZ^2(G,\muF)$. 
Given $x \in G$ not contained in $\rZ_c(G)$, pick $g \in G_\zz$ projecting to $x$, so that $g \not \in \rZ(G_\zz)$. 
Every cyclic group has a faithful irreducible representation. 
Therefore \Cref{thm:faithfulcenterpreservingirrep} applies to the pair $H = \langle g \rangle \leq G_\zz$, and yields an irreducible representation $\sigma$ of $G_\zz$ with $\rZ(\sigma) \cap H \leq \rZ(G_\zz)$. 
In particular, $g \not \in \qker(\sigma)$. 
Since $\muz \leq \Fi^\times$, Schur's lemma implies that the central kernel $\muz$ acts via $\sigma$ by scalar operators. 
This means that $\sigma$ factors through a projective representation $\ol \sigma$ of $G = G_\zz / \muz$ such that $x \notin \ker(\ol \sigma)$. 
Thus $\rK_c(G) \leq \rZ_c(G)$. 

Now assume that $\muF$ is $|G|$-divisible, and pick the representative $\zz$ of $c$ according to \Cref{lem:orderofcocyles} (with $A = \muF$). 
By the choice of $\zz$, the central kernel $\muz$ has order dividing $|G|$, hence the order of $G_\zz$ divides $|G|^2$. 
Any irreducible $c$-projective representation $\ol{\sigma}$ of $G$ lifts to an irreducible linear representation $\sigma$ of the finite central extension $G_{{\zz}}$. 
Since $\muF$ is $|G|$-divisible, Schur's lemma implies that the center $\rZ(G_\zz)$ acts via $\sigma$ by scalar operators lying in $\muF$. 
This shows that $\rZ_c(G) \leq \ker(\ol \sigma)$. 
\end{proof}

From here the following question is quite natural: 

\begin{ques} \label{ques:irred-c-faithful}
Given $c \in \rH^2(G,\Fi^\times)$, characterize whether $G$ admits an irreducible $c$-projective representation that is $c$-faithful.
\end{ques}

In the case of the trivial class $c=1$ over a $|G|$-closed field $\Fi$, observe that $G$ possesses an irreducible $1$-faithful $1$-projective representation $\ol \sigma$ if and only if $G$ has an irreducible representation $\sigma$ such that $\rZ(\sigma) = \rZ(G)$, that is, $G$ possesses a center-preserving irreducible representation $\sigma$. 
Equivalently, $G$ possesses an irreducible $1$-faithful $1$-projective representation if and only if $G$ has a {center-preserving} central quotient admitting a faithful irreducible representation. 
(Note that a central quotient map need not be center-preserving, hence the inflation of a faithful irreducible representation of a central quotient need not be center-preserving, as shown by \Cref{exe:D4C4}.) 
Obviously, it suffices that $G$ itself has a faithful irreducible representation, but that condition is not necessary. 
We will now work towards answering \Cref{ques:irred-c-faithful} under certain suitable assumptions on $\Fi$. 
\smallskip

Recall that $\rZ_2(G)$ denotes the second center of $G$, that is, the second term of the ascending central series of $G$.\footnote{There should be no confusion with the notation $\rZ_c(G)$ introduced previously, as we will only use the integer $1$ to denote the trivial cohomology class, and we will avoid to write $\rZ_1(G)$ for the (first) center of $G$ (even though they coincide anyways). } 
The commutator map defines a bi-homomorphism 
\[
\rZ_2(G) \times G \to \rZ(G) : (z_2, g) \mapsto [z_2, g]
\]
that factors through the bi-homomorphism
\[
\rZ_2(G)/\rZ(G) \times G/[G,G] \to \rZ(G),
\]
since $[z_2,g]$ depends only on the cosets $z_2\rZ(G)$ and $g[G,G]$. 
By post-composing with a fixed character $\chi \in \widehat{\rZ(G)}$ defined on the center of $G$, we obtain a group homomorphism 
\begin{align*}
\omega_\chi \colon \, \rZ_2(G)/\rZ(G) &\to \widehat{G/[G,G]} \\ 
z_2\rZ(G) & \mapsto \big(g[G,G] \mapsto \chi([z_2,g])\big).
\end{align*}

\begin{lem}\label{lem:centerpreserving}
Let $\chi \in \widehat{\rZ(G)}$ and set $K = \ker(\chi) \leq \rZ(G)$. 
The following conditions are equivalent. 
\begin{enumerate}[leftmargin=2em,itemsep=.2ex,label=\textup{(\roman*)}]
    \item The quotient map $\pi \colon G \to G/K$ is center-preserving. 
    \item The homomorphism $\omega_\chi \colon \ \rZ_2(G)/\rZ(G) \to \widehat{G/[G, G]}$ is injective.
\end{enumerate}
\end{lem}
\begin{proof}
Suppose that $\pi$ is center-preserving. 
Let $z_2 \in \rZ_2(G)$ be such that $z_2\rZ(G)$ belongs to the kernel of $\omega_\chi$. 
This means that $\chi([z_2, g]) = \ee$ for all $g \in G$, or in other words, that $[z_2, G] \leq \ker(\chi)$. 
Thus we have $[z_2, G] \leq K$, so that $\pi(z_2)$ is central in $G/K$. 
As $\pi$ is center-preserving, it follows that $z_2 \in \rZ(G)$. 
This proves that the kernel of $\omega_\chi$ is trivial. 

Conversely, assume that $\omega_\chi$ is injective. 
We will show that the quotient map $\pi \colon G \to G/K$ is center-preserving. 
Let $z_2 \in G$ be such that $\pi(z_2) \in \rZ(G/K)$. 
Then $[z_2, G] \leq K \leq \rZ(G)$, therefore $z_2 \in \rZ_2(G)$ and $\chi([z_2, g]) = \ee$ for all $g \in G$. 
This means that $z_2\rZ(G)$ is in the kernel of $\omega_\chi$, hence is trivial. 
In other words, $z_2 \in \rZ(G)$, proving the lemma. 
\end{proof}

Combining \Cref{lem:centerpreserving} with Gasch\"utz' \Cref{thm:Gaschutz}, we obtain the following criterion for the existence of a center-preserving irreducible representation.
This provides an answer to \Cref{ques:irred-c-faithful} for the trivial class $1 \in \rH^2(G, \Fi^\times)$, at least when $\Fi$ is $|G|$-closed. 

\begin{cor}[{of \Cref{thm:Gaschutz}}] \label{cor:centerpreserving}
The following conditions are equivalent. 
\begin{enumerate}[leftmargin=2em,itemsep=.4ex,label=\textup{(\roman*)}]
    \item $G$ possesses an irreducible representation $\sigma$ with $\rZ(\sigma) = \rZ(G)$. 
    \item There exists a character $\chi \in \widehat{\rZ(G)}$ such that:
    \begin{enumerate}[label=\textup{(\arabic*)}]
        \item \label{item:centerpreserving1} $\omega_\chi \colon \, \rZ_2(G)/\rZ(G) \to \widehat{G/[G,G]}$ is injective, and 
        \item \label{item:centerpreserving2} $\SocA(G/\ker(\chi))$ is generated by a single conjugacy class. 
\end{enumerate}
\end{enumerate}   
If in addition $G$ is nilpotent, then those conditions are equivalent to:
\begin{enumerate}[resume*]
\item There exists a character $\chi \in \widehat{\rZ(G)}$ such that $\omega_\chi $ is injective.
\end{enumerate}
\end{cor}
We shall see in \Cref{exe:quasikernel} that a non-nilpotent group may satisfy \Cref{item:centerpreserving1} from \Cref{cor:centerpreserving}, but not \Cref{item:centerpreserving2}, and thus fail to admit any irreducible center-preserving representations. 

\begin{proof}[Proof of \Cref{cor:centerpreserving}]
Suppose that (i) holds, and set $K = \ker(\sigma)$. 
Since $\sigma$ is center-preserving, we have a fortiori $K \leq \rZ(G)$. 
Moreover, as $G/K$ has a faithful irreducible representation, the image $\rZ(G)/K = \rZ(G/K)$ of the center of $G$ is cyclic. 
In consequence, there exists a character $\chi \in \widehat{\rZ(G)}$ whose kernel is precisely $K$. 
Because $\sigma$ is center-preserving, so is the quotient map $G\to G/K$. 
Hence $\omega_\chi$ is injective by \Cref{lem:centerpreserving}. 
As $G/K$ has a faithful irreducible representation, it then follows from \Cref{thm:Gaschutz} that $\SocA(G/K)$ is generated by a single conjugacy class. 
Thus (ii) holds. 

Conversely, given $\chi$ as in (ii), the quotient map $G \to G/ \ker(\chi)$ is center-preserving by \Cref{lem:centerpreserving}, and the quotient $G/ \ker(\chi)$ has a faithful irreducible representation by \Cref{thm:Gaschutz}. 
The inflation of that representation to $G$ is an irreducible center-preserving representation, as required by (i). 

It remains to show that when $G$ is nilpotent, \Cref{item:centerpreserving2} is a consequence of \Cref{item:centerpreserving1}. 
Assume thus that $G$ is nilpotent, so that the quotient $G/K$, for $K=\ker(\chi)$, is nilpotent as well. 
In consequence, every minimal normal subgroup of $G/K$ is central, hence $\SocA(G/K) \leq \rZ(G/K)$. 
If $\chi$ satisfies \Cref{item:centerpreserving1}, then by \Cref{lem:centerpreserving} we have 
\[
\rZ(G/K) = \pi(\rZ(G)) \cong \rZ(G)/K = \rZ(G)/\ker(\chi) \cong \chi(\rZ(G)).
\]
This shows that the center of $G/K$ is cyclic, hence so is $\SocA(G/K)$. 
In particular, \Cref{item:centerpreserving2} holds. 
\end{proof}

As can be expected, \Cref{ques:irred-c-faithful} about projective representations of $G$ can be reformulated as a question on linear representations of the central extension $G_c$, using the following observation. 

\begin{prop}\label{prop:G-G_c}
Assume that $\muF$ is $|G|$-divisible. 
For $c \in \rH^2(G,\muF)$, the following conditions are equivalent. 
\begin{enumerate}[leftmargin=2em,itemsep=.4ex,label=\textup{(\roman*)}]
    \item \label{item:G-G_c1} $G$ possesses an irreducible $c$-projective representation $\ol{\sigma}$ that is $c$-faithful.
    \item \label{item:G-G_c2} For some (hence every) $\zz \in \rZ^2(G,\muF)$ representing $c$, the finite group $G_\zz$ possesses an irreducible linear representation $\sigma$ such that $\rZ(\sigma) = \rZ(G_\zz)$. 
    \item \label{item:G-G_c3} $\tilde{G}_c$ possesses an irreducible representation $\tilde{\sigma}$ whose restriction identifies $\Fi^\times$ with the scalar operators, and is such that $\rZ(\tilde{\sigma}) = \rZ(\tilde{G}_c)$. 
\end{enumerate}
\end{prop}
\begin{proof}
First, we verify that assertion \ref{item:G-G_c2} is equivalent to \ref{item:G-G_c3}, and in particular, independent of the choice of $\zz \in \rZ^2(G,\muF)$. 
For $\tilde{\sigma}$ an irreducible representation of $\tilde{G}_c$ afforded by \ref{item:G-G_c3}, let $\sigma$ be the restriction of $\tilde{\sigma}$ to $G_\zz$, seen as a subextension of $\tilde{G}_c$. 
By construction, we have $\tilde{\sigma}(\tilde{G}_c) = F^\times \cdot \sigma(G_\zz)$, hence $\sigma$ is an irreducible representation of $G_\zz$. 
For the same reason, if $g \in \rZ(\sigma)$, then $g \in \rZ(\tilde{\sigma})$. 
This implies that $\rZ(\sigma) = \rZ(\tilde{\sigma}) \cap G_\zz = \rZ(\tilde{G}_c) \cap G_\zz \leq \rZ(G_\zz)$. 

Conversely, let $\sigma$ be an irreducible representation of $G_\zz$ satisfying $\rZ(\sigma) = \rZ(G_\zz)$. 
We extend $\sigma$ to a representation $\tilde{\sigma}$ of $\tilde{G}_c$ as follows. 
Any $g \in \tilde{G}_c$ writes as a product $g = \lambda h$ with $\lambda \in F^\times \leq \tilde{G}_c$ and $h \in G_\zz$. 
We set $\tilde{\sigma}(g) = \lambda \sigma(h)$, where this time $\lambda$ is seen as the corresponding scalar operator in the representation $\sigma$. 
This assignment does not depend on the choice of $\lambda$ and $h$: if $\lambda h = \lambda' h' \in \tilde{G}_c$, then $\lambda'^{-1} \lambda = h'h^{-1} \in \Fi^\times \cap G_\zz = \muz$, hence
\[
\lambda \sigma(h) = \lambda \sigma( (\lambda^{-1} \lambda') h' h^{-1} h) = \lambda (\lambda^{-1} \lambda') \sigma(h') = \lambda' \sigma(h').
\]
If now $g = \lambda h \in \rZ(\tilde{\sigma})$, then $\sigma(h) \in \rZ(\sigma(G_\zz))$, so that $h \in \rZ(\sigma) = \rZ(G_\zz) \leq \rZ(\tilde{G}_c)$ and in turn, $g \in \rZ(\tilde{G}_c)$. 

\medbreak
Next, let $\ol \sigma$ be an irreducible $c$-faithful $c$-projective representation of $G$. 
By \Cref{lem:orderofcocyles}, there is a representative $\zz$ of $c$ in $\rZ^2(G,\muF)$ for which $|G_\zz|$ divides $|G|^2$. 
The lift $\sigma$ of $\ol \sigma$ to $G_\zz$ is an irreducible linear representation, and by hypothesis $\ker(\ol{\sigma})=\rZ_c(G)$ is the image of $\rZ(G_\zz)$. 
By Schur's lemma and the assumption on $\muF$, it follows that $\rZ(\sigma)$ acts by scalars via $\sigma$. 
This means that $\pi(\rZ(\sigma)) \leq \ker(\ol{\sigma}) = \rZ_c(G)$, where $\pi \colon G_\zz \to G$ is the canonical quotient map. 
Since $\ker(\pi)$ is contained in $\rZ(G_\zz) \leq \rZ(\sigma)$, we infer that $\rZ(\sigma) = \rZ(G_\zz)$. 
Thus \ref{item:G-G_c1} implies \ref{item:G-G_c2}. 

The proof of the converse is similar: with $\zz$ as above, an irreducible linear representation $\sigma$ of $G_\zz$ satisfying $\rZ(\sigma) = \rZ(G_\zz)$ descends by Schur's lemma to a $c$-projective representation $\ol \sigma$ of $G$ with $\ker(\ol \sigma) = \pi (\rZ(\sigma)) = \pi(\rZ(G_\zz)) = \rZ_c(G)$, showing that \ref{item:G-G_c2} implies \ref{item:G-G_c1}. 
\end{proof}

Combining \Cref{cor:centerpreserving} with \Cref{prop:G-G_c}, we obtain an answer to \Cref{ques:irred-c-faithful}, at least for projective representations over fields whose multiplicative group is $|G|$-divisible.

\subsection{Examples} \label{subsec:examplesfpr}
Before addressing the relative case, let us point out that an answer to \Cref{ques:irred-c-faithful} cannot only depend on the pair $(G, \rZ_c(G))$. 
Indeed, the following example shows that there can exist different classes $c_1, c_2 \in \rH^2(G,\Fi^\times)$ with $\rZ_{c_1}(G) = \rZ_{c_2}(G)$, and such that $G$ has an irreducible $c_1$-faithful $c_1$-projective representation, but no irreducible $c_2$-faithful $c_2$-projective representation. 

\begin{exe} \label{exe:dependenceonclass}
Let $p$ be an odd prime, let $V = \F_p^2$, and let
\[
G = \big(V^3 \rtimes \SL(V)\big) \times V,
\]
where $\SL(V)$ acts diagonally on $V^3$. 
(Observe that $V^3 \rtimes \SL(V)$ is a perfect, centerless group that does not have any faithful irreducible representation, since the $\SL(V)$-module $V^3$ is not cyclic.) 
Let also $H$ be the Heisenberg group over $\F_p$. 

We form two different central extensions of $G$ by a cyclic group of order $p$ as follows. The first extension is 
\[
G_1 = \big(V^3 \rtimes \SL(V)\big) \times H.
\]
To define the second extension, we first observe that $\SL(V)$ has a natural action by automorphisms on $H$, that is trivial on $\rZ(H)$ and descends to the natural $\SL(V)$-action on $V \cong H/\rZ(H)$. 
Let then 
\[
L = \big(H^3 \rtimes \SL(V)\big) \times H,
\]
where $\SL(V)$ acts diagonally on $H^3$, and let $G_2$ be the central quotient of $L$ obtained by identifying the centers of the four factors isomorphic to $H$ in $L$. 
By construction, the center of $G_2$ is cyclic of order $p$. 

Suppose that $\Fi$ contains $p$ distinct $p$\textsuperscript{th} roots of unity. 
We may then view $G_1$ and $G_2$ as central extensions of $G$, associated with respective cohomology classes $c_1, c_2 \in \rH^2(G, F^\times)$. 
Since $\rZ(G_1)$ and $\rZ(G_2)$ are both cyclic of order $p$, we have $\rZ_{c_1}(G) = \rZ_{c_2}(G) = \{\ee\}$. 
Moreover, we have $\SocA(G_1) = V^3 \times \rZ(H)$; it is not generated by a single conjugacy class, because the $\SL(V)$-module $V^3$ is not cyclic (see \cite[Lemma~3.4]{CH20}). 
It follows that for any central subgroup $K \leq \rZ(G_1)$, the normal subgroup $\SocA(G_1/K)$ is not generated by a single conjugacy class. Therefore $G_1$ does not have any irreducible representation $\sigma$ with $\rZ(\sigma) = \rZ(G_1)$. 
On the other hand, the group $\SocA(G_2) = \rZ(G_2)$ is cyclic, hence $G_2$ has a faithful irreducible representation. 
It follows from Proposition~\ref{prop:G-G_c} that, despite the fact that $\rZ_{c_1}(G) = \rZ_{c_2}(G) = \{\ee\}$, the group $G$ possesses an irreducible $c_2$-projective representation that is $c_2$-faithful, but no irreducible $c_1$-projective representation that is $c_1$-faithful. 
\end{exe}

\begin{exe} \label{exe:quasikernel}
We continue to consider the group $G_1 = \big(V^3 \rtimes \SL(V)\big) \times H$ defined in \Cref{exe:dependenceonclass} above. 
As mentioned there, the group $G_1$ does not have any central subgroup $K$ such that $\SocA(G_1/K)$ is generated by a single conjugacy class, because the $\SL(V)$-module $V^3$ is not cyclic. 
On the other hand, we have $\rZ(G_1) \cong \rZ(H) \cong C_p$ and $\rZ_2(G_1) \cong H$. 
It follows that for any faithful character $\chi$ of the cyclic group $\rZ(G_1) \cong C_p$, the homomorphism $\omega_\chi \colon \rZ_2(G_1)/\rZ(G_1) \to \widehat{G_1/[G_1, G_1]}$ is injective: indeed, it may be viewed as an isomorphism $H/\rZ(H) \to \widehat{H/[H, H]} \cong H/[H, H]$.
Thus $G_1$ satisfies \Cref{item:centerpreserving1} but not \Cref{item:centerpreserving2} in \Cref{cor:centerpreserving}.
\end{exe}

\subsection{The relative case} \label{subsec:relativefpr}
We conclude the article with an open-ended discussion concerning the projective analogue of representations that are center-preserving on a subgroup. 
This entails more complications than in the linear case; they seem unavoidable, especially when $\Fi$ is not $|G|$-closed. 

\medskip
Let $H \leq G$ be a subgroup. 
Every cohomology class $c \in \rH^2(G,\Fi^\times)$ defines a class $\restr{c}{H} \in \rH^2(H,\Fi^\times)$ by restriction. 

\Cref{cor:centerpreservingprojirrep} can be extended in the following way to classes in $\rH^2(G,\muF)$ whose restriction to $H$ is trivial. 

\begin{cor}[of \Cref{thm:faithfulcenterpreservingirrep}] \label{cor:centerpreservingprojirrep2}
Let $H \leq G$, and suppose that $H$ possesses a faithful irreducible representation $\rho$. 
Let $c \in \rH^2(G,\muF)$ be such that $\restr{c}{H} = 1$. 
Then there exists an irreducible $c$-projective representation $\ol{\sigma}$ of $G$ such that $\ker(\ol{\sigma}) \cap H \leq \rZ_c(G)$ and $\restr{\ol{\sigma}}{H}$ contains the projectivization of $\rho$. 

If in addition $\rZ_c(G) \cap H = \{\ee\}$, then $\ol{\sigma}$ is an irreducible $c$-projective representation of $G$ which is faithful on $H$. 
\end{cor}

\begin{proof}
Let $\zz$ be a representative of $c$ in $\rZ^2(G,\muF)$. 
Recall that the central extension $G_\zz$ of $G$ by $\muz$ is a finite group to which every $c$-projective representation of $G$ lifts. 
Because $\restr{c}{H} = 1$, the preimage $H_\zz$ of $H$ in $G_\zz$ is a split extension of $H$. 
Thus, $H$ is a subgroup of $G_\zz$. 

By \Cref{thm:faithfulcenterpreservingirrep}, $G_\zz$ admits an irreducible representation $\sigma$ such that $\rZ(\sigma) \cap H \leq \rZ(G_\zz)$ and $\restr{\sigma}{H}$ contains $\rho$. 
Since $\muz \leq F^\times$, Schur's lemma implies that $\muz$ acts by scalars under $\sigma$. 
In consequence, the projectivization $\ol \sigma$ of $\sigma$ is indeed a $c$-projective representation of $G$, whose restriction $\restr{\ol{\sigma}}{H}$ contains the projectivization of $\rho$ by construction. 
As $\ker(\ol \sigma) \cap H$ is the image of $\qker(\sigma) \cap H \leq \rZ(\sigma) \cap H$, we deduce that $\ker(\ol \sigma) \cap H \leq \rZ_c(G) \cap H$. 
\end{proof}

Our main theorem and these two corollaries suggest the following question for projective representations. 

\begin{ques}\label{ques:centerpreservingprojirrep}
Let $H \leq G$ be a subgroup, and let $d \in \rH^2(H, F^\times)$. 
Suppose that $H$ has an irreducible $d$-faithful $d$-projective representation $\ol \rho$. 
Can one characterize the existence of a class $c \in \rH^2(G,\Fi^\times)$ restricting to $\restr{c}{H} = d$ and an irreducible $c$-projective representation $\ol \sigma$ of $G$, such that $\ker(\ol \sigma) \cap H \leq \rZ_c(G)$? 
\end{ques}

In case $H$ is abelian, and $d = 1$, we have $\rZ_d(H) = \rZ(H) = H$, so that the hypothesis on $H$ is automatic when for example $\Fi$ is $|H|$-closed. 
\Cref{exe:D4D4D4} then shows that one cannot expect a positive answer to this question for every cohomology class $c \in \rH^2(G, F^\times)$ that restricts to $d$. 
Indeed, it exhibits a pair $H \leq G$ for which every irreducible $1$-projective representations $\ol \sigma$ of $G$ satisfies $\ker(\ol \sigma) \cap H \neq \{\ee\}$, while $\rZ(G) \cap H = \{\ee\}$. 

\begin{rem} \label{rem:centerpreservingprojirrep}
Under the alternate hypothesis that $H$ possesses a faithful irreducible $\restr{c}{H}$-projective representation $\ol \rho$, \Cref{cor:centerpreservingprojirrep2} can be extended to every class $c \in \rH^2(G,\muF)$ as follows: there exists an irreducible $c$-projective representation $\ol{\sigma}$ of $G$ such that $\ker(\ol \sigma) \cap H \leq \rZ_c(G)$ and $\restr{\ol{\sigma}}{H}$ contains $\ol \rho$. 
If in addition $\muF$ is $|G|$-divisible, or $\rZ_c(G) \cap H = \{\ee\}$, then in fact $\ol{\sigma}$ is faithful on $H$. 
We leave the details to the reader; the argument is similar to the proof of \Cref{cor:centerpreservingprojirrep2}. 
It turns out that this statement is superseded by the next proposition, whose proof relies on our endmost application of Frobenius reciprocity. 
\end{rem}

\begin{prop} \label{prop:inducefpr}
Let $H \leq G$, and fix a class $c \in \rH^2(G,\Fi^\times)$. 
Suppose that $H$ possesses a faithful irreducible $\restr{c}{H}$-projective representation $\ol{\rho}$. 
Then there exists an irreducible $c$-projective representation $\ol{\sigma}$ of $G$ whose restriction to $H$ contains $\ol{\rho}$, hence is faithful on $H$. 
\end{prop}
\begin{proof}
Set $d = \restr{c}{H}$. 
Recall that $\tilde{G}_c$ (resp.\ $\tilde{H}_d$) denotes the central extension of $G$ (resp.\ $H$) determined by the class $c$ (resp.\ $d$), so that we have the following homomorphism of central extensions:
\[\begin{tikzcd}
	1 & {\Fi^\times} & {\tilde{G}_c} & G & 1 \\
	1 & {\Fi^\times} & {\tilde{H}_d} & H & 1
	\arrow[from=1-1, to=1-2]
	\arrow[from=1-2, to=1-3]
	\arrow[from=1-3, to=1-4]
	\arrow[from=1-4, to=1-5]
	\arrow[from=2-1, to=2-2]
	\arrow[equal, from=2-2, to=1-2]
	\arrow[from=2-2, to=2-3]
	\arrow[hook, from=2-3, to=1-3]
	\arrow[from=2-3, to=2-4]
	\arrow[hook, from=2-4, to=1-4]
	\arrow[from=2-4, to=2-5]
\end{tikzcd}\]
By construction, $\ol{\rho}$ lifts to a representation $\rho$ of $\tilde{H}_c$, for which the subgroup of scalar transformations is precisely $\Fi^\times \leq \tilde{H}_d$ because $\ol{\rho}$ is faithful. 
Pick an irreducible constituent $\sigma$ of $\textstyle \ind_{\tilde{H}_d}^{\tilde{G}_c} \rho$. 
By definition of the induced representation, the central kernel $\Fi^\times \leq \tilde{G}_c$ acts by scalars under $\textstyle \ind_{\tilde{H}_d}^{\tilde{G}_c} \rho$, hence also under $\sigma$. 
By Frobenius reciprocity, $\rho$ is contained in the restriction of $\sigma$ to $\tilde{H}_d$. 
This shows that the projectivization $\ol{\sigma}$ of $\sigma$ is an irreducible $c$-projective representation of $G$, whose restriction to $H$ contains $\ol{\rho}$. 
\end{proof}

\section*{Acknowledgments}
We kindly thank P.\ de la Harpe, G.\ Navarro, J.-P.\ Serre, and A.\ Valette for their helpful comments and corrections, that improved the quality of this paper.

%%%%% Bibliography %%%%%

\newpage

\bibliographystyle{alphaabbrv}
\bibliography{references}
    
\end{document}